\documentclass[11pt]{amsart}

\usepackage[T1]{fontenc}
\usepackage{times}
\usepackage{float}

\usepackage[dvips]{graphics,color}
\usepackage{graphicx}
\usepackage{amsmath,amsthm,amssymb,amscd}
\usepackage{amsfonts}
\usepackage{amsmath,amssymb,amsthm,enumerate,epsfig,graphicx}
\usepackage{ifthen,latexsym,syntonly}
\usepackage{natbib}
\usepackage{amsmath}
\usepackage{amssymb}
\usepackage{verbatim}
\newtheorem{thm}{Theorem}
\newtheorem{cor}{Corollary}

\newtheorem{dfn}{Definition}

\newtheorem{lemma}{Lemma}
\newtheorem{prop}{Proposition}
\newtheorem{asmptn}{Assumption}

\newtheorem{claim}{Claim}

\newtheorem{watanabe}{Watanabe's Theorem}

\newtheorem{sugitani}{Sugitani's Theorem}

\theoremstyle{remark}
\newtheorem{remark}{Remark}

\newcommand{\zz}[1]{\mathbb{#1}}
\newcommand{\la}{\langle}
\newcommand{\ra}{\rangle}
\newcommand{\ip}[1]{\langle #1 \rangle}

\begin{document}

\title[BRW Local Time]{Spatial Epidemics and Local Times for Critical Branching Random
Walks in Dimensions 2 and 3}
\author{Steven P. Lalley}
\address{Department of Statistics, The University of Chicago, Chicago, IL
60637} \email{lalley@galton.uchicago.edu}

\author{Xinghua Zheng}
\address{Department of Statistics, The University of Chicago,
Chicago, IL 60637} \curraddr{Department of Mathematics, The
University of British Columbia\\ Room 121, 1984 Mathematics Road
Vancouver, B.C., Canada V6T 1Z2} \email{xhzheng@math.ubc.ca}

\date{\today}

\maketitle

\renewcommand{\baselinestretch}{1.0}
\normalsize

\begin{abstract}
The behavior at criticality of spatial \emph{SIR}
(susceptible/infected/recovered) epidemic models in dimensions two
and three is investigated. In these models, finite populations of
size $N$ are situated at the vertices of the integer lattice, and
infectious contacts are limited to individuals at the same or at
neighboring sites. Susceptible individuals, once infected, remain
contagious for one unit of time and then recover, after which they
are immune to further infection.  It is shown that the
measure-valued processes associated with these epidemics, suitably
scaled, converge, in the large-$N$ limit, either to a standard
Dawson-Watanabe process (super-Brownian motion) or to a
Dawson-Watanabe process with location-dependent killing, depending
on the size of the the initially infected set. A key element of
the argument is a proof of Adler's 1993 conjecture that the local
time processes associated with branching random walks converge to
the local time density process associated with the limiting
super-Brownian motion.

\end{abstract}

\section{Introduction}
\subsection{Spatial \emph{SIR} epidemics }
Simple spatial models of epidemics are known to exhibit \emph{critical
thresholds} in one dimension: Roughly, when the density of the
initially infected set exceeds a certain level, the epidemic evolves
in a markedly different fashion than its \emph{branching
envelope}. See \cite{lalley07} for a precise statement, and
\cite{Aldous97}, \cite{Martinlof98}, and \cite{LD06} for analogous
results in the simpler setting of mean-field models.  The main purpose of
this article is to show that spatial \emph{SIR} epidemics (\emph{SIR}
stands for \emph{susceptible/infected/recovered}) in dimensions two
and three also exhibit critical thresholds.

The epidemic models studied here take place in populations of size $N$
located at the sites of the integer lattice $\zz{Z}^{d}$ in $d$
dimensions.  Each of the $N$ individuals at a site $x\in \zz{Z}^{d}$
may at any time be either susceptible, infected, or
recovered. Infected individuals remain infected for one unit of time,
and then recover, after which they are immune to further infection.
The rules of infections are as follows: at each time $t=0,1,2,\ldots,$
for each pair $(i_x, s_y)$ of an infected individual located at $x$
and a susceptible individual at $y$, $i_x$ infects $s_y$ with
probability $p_N(x,y).$ We shall only consider the case where the
transmission probabilities $p_N(x, y)$ are spatially homogeneous,
nearest-neighbor, and symmetric, and scale with the village size $N$
in such a way that the expected number of infections by a contagious
individual in an otherwise healthy population is 1 (so that the
epidemic is critical), that is,
\begin{asmptn}
$p_N(x; y) = 1/[(2d+1)N]$ if $|y- x|\leq  1$; and = 0 otherwise.
\end{asmptn}

Our main result, Theorem~\ref{thm:SIR} below, asserts that under
suitable hypotheses on the initial configurations of infected
individuals, the critical spatial \emph{SIR}$-d$ epidemic can be
rescaled so as to converge to a Dawson-Watanabe measure-valued
diffusion in both $d=2$ and $d=3$. Depending on the size of the
initially infected set, the limiting Dawson-Watanabe process has
either a positive killing rate or no killing at all.  The
analogous result for $d=1$ was proved in \cite{lalley07}, using
the fact that one-dimensional super-Brownian motion (the
Dawson-Watanabe process with no killing) has sample paths in the
space of absolutely continuous measures. In higher dimensions this
is no longer true, so a different strategy is needed.

\subsection{Branching envelope of a spatial epidemic} \label{ssec:bre}
The spatial \emph{SIR} epidemic in $d$ dimensions is naturally
coupled with a nearest neighbor branching random walk on the
integer lattice $\zz{Z}^d$; this branching random walk is often
referred to as the \emph{branching envelope} of the epidemic.
Particles of this branching random walk represent \emph{infection
attempts} in the coupled epidemic, some of which may fail to be
realized in the epidemic because the targets of the attempts are
either recovered or are targets of other simultaneous infection
attempts. The branching envelope evolves as follows: Any particle
located at site $x$ at time $t$ lives for one unit of time and
then reproduces, placing random numbers $\xi_y$ of offspring at
the sites $y$ such that $|y- x|\leq 1$. The random variables
$\xi_y$ are i.i.d., with Binomial$(N,1/[(2d+1)N])$ distributions.
Denote this reproduction rule by $\mathcal{R}_N$, and denote by
$\mathcal{R}_\infty$ the corresponding offspring law in which the
Binomial distribution is replaced by the Poisson distribution with
mean $1/(2d+1)$. Since offspring are placed independently at each
of the $2d + 1$ nearest neighbors, the expected total number of
offspring of a particle is 1, \hbox{i.e.} the branching random
walk is critical. Moreover, under either reproduction rule
$\mathcal{R}_{N}$ or $\mathcal{R}_{\infty}$, particle motion is
governed by the law of the simple nearest neighbor random
walk\footnote{Throughout the paper, the term \emph{simple random
walk} will mean simple random walk with holding probability $1/
(2d+1)$.} on $\zz{Z}^{d}$ (with holding probability $1/ (2d+1)$):
In particular, given that a particle at site $x$ has $k$
offspring, each of these offspring independently chooses a
neighboring site $y$ according to the law
\begin{equation}\label{eq:srw}
    P_{1} (x,y)= 1/ (2d+1) \quad \text{for} \quad  |y-x|\leq 1.
\end{equation}
Note that the covariance matrix of the increment has determinant
$\sigma^{2d}$, where  $\sigma^{2}$ is the \emph{variance
parameter} of the jump distribution, defined by
\begin{equation}\label{eq:sigma}
    \sigma^2:= \left(\frac{2}{2d+1} \right).
\end{equation}

The spatial \emph{SIR}-$d$ epidemic can be constructed together
with its branching envelope on a common probability space in such
a way that the branching envelope dominates the epidemic, that is,
for each time $n$ and each site $x$ the number of infected
individuals at site $x$ at time $n$ is no larger than the number
of particles in the branching envelope.  The construction, in
brief, is as follows (see \cite{lalley07}): Particles of the
branching random walk will be colored either \emph{red} or
\emph{blue} according to whether or not they represent infections
that actually take place, with red particles representing actual
infections. Initially, all particles are red.  At each time
$t=0,1,2,\dotsc$, particles produce offspring at the same or
neighboring sites according to the law $\mathcal{R}_{N}$ described
above. Offspring of blue particles are always blue, but offspring
of red particles may be either red or blue, with the choices made
according to the following procedure: All offspring of red
particles at a location $y$ choose numbers $j\in [N]:=\{1,2,\dotsc
,N \}$ at random, independently of all other particles.  If a
particle chooses a number $j$ that was previously chosen by a
particle of an earlier generation at the same site $y$, then it is
assigned color blue. If $k>1$ offspring of red particles choose
the same number $j$ at the same time, and if $j$ was not chosen in
an earlier generation, then $1$ of the particles is assigned color
red, while the remaining $k-1$ are assigned color blue. Under this
rule, the subpopulation of red particles evolves as an
\emph{SIR}$-d$ epidemic.

It is apparent that when the numbers of infected and recovered
individuals at a site and its nearest neighbors are small compared
to $N$, then blue particles will be produced only infrequently,
and so the epidemic process will closely track its branching
envelope. Only when the sizes of the recovered and infected sets
reach certain critical thresholds will blue particles start to be
produced in large numbers, at which point the epidemic will begin
to diverge significantly from the branching envelope. Our main
result, Theorem~\ref{thm:SIR}  below, implies that the critical
threshold for the number of initially infected individuals is on
the order $N^{1/ (3-d/2)}$.

The \emph{SIR}-$d$ epidemic is related to its branching envelope
in a second --- and for our purposes more important -- way. The
law of the epidemic, as a probability measure on the space of
possible population trajectories, is absolutely continuous
relative to the law of its branching envelope. The likelihood
ratio can be expressed as a product over time and space, with each
site/neighbor/generation contributing a factor (see
\S~\ref{ssec:convergenceLR} below). Each such factor involves the
\emph{total occupation time} $R_{n}^{N} (x)$ of the site, that is,
the sum of the number of particles at site $x$ over all times
prior to $n$. Thus, the asymptotic behavior of the occupation time
statistics for branching random walks will play a central role in
the analysis of the large-$N$ behavior of the \emph{SIR}-$d$
epidemic.

\subsection{Watanabe's Theorem}\label{ssec:watanabe} A fundamental
theorem of \cite{watanabe68} asserts that, under suitable rescaling
(the Feller scaling) the measure-valued processes naturally associated
with critical branching random walks converge to a limit, the standard
Dawson-Watanabe process, also known as super-Brownian motion.

\begin{dfn} The Feller-Watanabe scaling operator $\mathcal{F}_k$
scales mass by $1/k$ and space by $1/\sqrt{k}$, that is, for any
finite Borel measure $\mu(dx)$ on $\zz{R}^d$ and any test function
$\psi \in C_{c}^{\infty } (\zz{R}^{d})$,
\begin{equation}\label{eqn:feller_scl}
    \la \psi,\mathcal{F}_k \mu \ra = k^{-1}\int \psi(\sqrt{k}x)\mu(dx).
\end{equation}
\end{dfn}

\begin{watanabe} Fix $N$, and for each $k = 1, 2,
\ldots,$ let $X^k_t$ be a branching random walk with offspring
distribution $\mathcal{R}_N$ and initial particle configuration
$X^k_0$. (In particular, $X^k_t (x)$ denotes the number of
particles at site $x\in\zz{Z}^d$ in generation $[t]$, and $X^k_t$
is the corresponding counting measure.) If the initial
mass distributions converge, after rescaling, as $k\to \infty$,
that is, if
\begin{equation}\label{ass:mu_conv}
    \mathcal{F}_k X^k_0 \Rightarrow \mu =X_0
\end{equation}
 for some finite Borel measure
$\mu $ on $\zz{R}^d$, then the rescaled measure-valued processes
$(\mathcal{F}_k X^k)_{kt}$ converge in law as $k\to\infty$:
$$
  (\mathcal{F}_k X^k)_{kt} \Rightarrow X_t,
$$
where $\Rightarrow$ represents the weak convergence relative to
the Skorokhod topology on\\ $D([0,\infty);M_F(\zz{R}^d))$. The
limit is the standard Dawson-Watanabe process $X_t$
(super-Brownian motion) with variance parameter $\sigma ^{2}$
(equivalently, standard super-Brownian motion run at speed
$\sigma$).
\end{watanabe}

 See \cite{etheridge} for
an in-depth study of the Dawson-Watanabe process and a detailed proof
of Watanabe's Theorem. Because the process $X_{t}$ has continuous
sample paths in the space of finite Borel measures, it follows
routinely from Watanabe's theorem that the occupation measures for
branching random walks converge to those of super-Brownian motion:

\begin{lemma}\label{lemma:jtconv_sbm_lt} The following joint
convergence holds:
\begin{equation}\label{eq:weakConvergence}
\left((\mathcal{F}_k X^k)_{kt}, \left(\int_0^t (\mathcal{F}_k X^k)_{ks}\, ds\right)\right)
   \Rightarrow \left(X_t,\int_0^t X_s \,ds\right),
\end{equation}
where $\Rightarrow$ represents  weak convergence relative to the
Skorokhod topology on $D([0,\infty);M_F(\zz{R}^d))^2$.
\end{lemma}

\begin{proof}The Dawson-Watanabe process $X_t$ has continuous sample paths in
$D([0,\infty);M_F(\zz{R}^d))$, see, e.g., Proposition 2.15 in
\cite{etheridge}. The functional $(X_t)\mapsto (\int_0^t X_s\,
ds)$ is continuous  relative to the Skorokhod topology on the
subspace of continuous measure-valued processes, so the result
follows from Watanabe's theorem and the continuous mapping
principle.
\end{proof}

\subsection{Local times of critical branching random walks}\label{ssec:lts}

In dimension $d = 1$ the super-Brownian motion has sample paths in
the space of absolutely continuous measures, that is, for each
$t>0$ the random measure $X_t$ is absolutely continuous relative
to Lebesgue measure (\cite{ks88}, \cite{Reimers89}). Moreover, the
Radon-Nikodym derivative $X(t, x)$ is jointly continuous in $t, x$
(for $t > 0$). It is shown in \cite{lalley07} that if a sequence
of branching random walks satisfy the assumptions in Watanabe's
Theorem, then the density processes associated with those
branching random walks, under some smoothness assumptions on the
initial configurations and after suitable scaling, converge to the
density process of the limiting super-Brownian motion.

In dimensions $d\geq 2$ the measure $X_t$ is almost surely singular
(\cite{DH79}). Therefore, one cannot expect the convergence of density
processes as in \cite{lalley07}. We shall prove, however, that the
\emph{occupation measures} of critical branching random walks have
discrete densities that converge weakly --- see Theorem \ref{thm:lt}
below. The limit process is the \emph{local time} process associated
with the occupation measure
\[
    L_{t}:=\int_{0}^{t} X_{s}\,ds
\]
of the super-Brownian motion. In dimensions $d=2,3$, the random
measure $L_{t}$ is, for each $t>0$, absolutely continuous, despite
the fact that $X_{t}$ is singular --- see \cite{Sugitani89},
\cite{Iscoe86b} and \cite{Fleischmann88}.  Moreover, under
suitable hypotheses on the initial condition $X_{0}$, the density
process $L_{t} (x)$ is jointly continuous for $t>0$ and $x\in
\zz{R}^{d}$: This is the content of \emph{Sugitani's theorem}. For
the reader's convenience, we state Sugitani's Theorem precisely
here. For $t>0$ and $x\in \zz{R}^{d}$, set
\[
    q_t(x)=\int_0^t \phi_s(x)\,ds,
    \quad \text{where} \quad
    \phi_{t} (x) =\frac{e^{-|x|^{2}/2t}}{(2\pi t)^{d/2}}
\]
is the usual heat kernel.

\begin{sugitani} Assume that $d=2$ or $3$, and that the
initial configuration $\mu:=X_0$ of the super-Brownian motion $X_t$
is such that  the convolution
\begin{equation}
  \label{ass:mu_smooth} ( q_t * \mu )(x)\text{ is jointly continuous in
  } t\geq 0 \text{ and }x\in\zz{R}^d.
\end{equation}
Then for each $t\geq 0$ the occupation measure $L_{t}$ is absolutely
continuous, and there is a jointly continuous version $L_{t} (x)$ of
the density process.
\end{sugitani}

We call $\left(L_t(x)\right)_{t\geq 0,\ x\in\zz{R}^d}$ the
\emph{local time density process} associated with the
super-Brownian motion. In view of Watanabe's and Sugitani's
theorems, it is natural to conjecture (see
Remark~\ref{remark:adler} below) that the local time density
processes of branching random walks, suitably scaled, converge to
the local time density process of the super-Brownian motion.
Theorem \ref{thm:lt} below asserts that this conjecture is true.
Let $X^{k}$ be a sequence of branching random walks on
$\zz{Z}^{d}$. Write
\begin{align}
\label{eq:Rn}
     X^k_i(x) :&=\# \text{ particles at }\; x \; \text{at time} \; i, \quad \text{and}\\
 \notag   R^k_n(x):&=\sum_{ i< n} X^k_i(x).
\end{align}
(We use the notation $R^{k}_{n}$ instead of $L^{k}_{n}$ because in
the corresponding spatial epidemic model, the quantity $R^{k}_{n}
(x)$ represents the number of \emph{recovered} individuals at site
$x$ and time $n$.) Denote by
\begin{equation}\label{eq:Pn}
    \zz{P}^{n} = (P_{n} (x,y))_{x,y\in \zz{Z}^{d}}= (P_{n} (y-x))_{x,y\in \zz{Z}^{d}}
\end{equation}
the transition probability kernel of the simple random walk on
$\zz{Z}^{d}$, that is, $\zz{P}^{n}=\zz{P}*\zz{P}^{n-1}$ is the
$n$th convolution power of the one-step transition probability
kernel  given by \eqref{eq:srw}. Let $G_{n} (x,y)$ be the
associated Green's function:
\[
    G_n(x):=\sum_{ i< n} P_i(x) .
\]
For any finite measure $\mu$ on $\zz{Z}^{d}$ with finite support,
set
\[
    (\mu G_n)(x):= (\mu * G_{n}) (x)=\sum_y \mu(y) G_n(x-y),
\]
and denote by $\mu G_{t} (y)$ the continuous extension to $[0,\infty )\times
\zz{R}^{d}$ by linear interpolation.

\begin{thm}\label{thm:lt} Assume that $d=2$ or $d=3$. For each $k = 1,
2, \ldots,$ let $X^k_t$ be a branching random walk whose offspring
distribution is Poisson with mean 1. Assume that the initial configurations
$\mu^k:=X^k_0$ satisfy hypothesis \eqref{ass:mu_conv} of Watanabe's
theorem, where the limit measure $\mu$
 has compact support  and satisfies the hypothesis
\eqref{ass:mu_smooth}  of Sugitani's theorem.  Assume further that
\begin{equation}\label{ass:lt_smooth}
    \frac{\mu^kG_{kt} (\sqrt{k}x)}{k^{2-d/2}}
    \Longrightarrow  [(q_{\sigma^2 t} *\mu)/\sigma^2] (x),
\end{equation}
where $\Rightarrow $ indicates weak convergence in the topology of
$D([0,\infty),C_b(\zz{R}^d))$.  Then as $k\to\infty$,
\begin{equation}\label{eq:lt}
    \frac{R^k_{kt}(\sqrt{k}x)}{k^{2-d/2} }
    \Longrightarrow
    L_{t} (x),
\end{equation}
where  $L_{t} (x)$ is the local time density process
 associated with the super-Brownian motion with variance parameter $\sigma ^{2}$ started in the initial
configuration $X_{0}=\mu$.
\end{thm}

Theorem \ref{thm:lt} will be proved in \S 2.

\begin{remark}
\label{remark:adler} The analogous result for critical branching
Brownian motions was conjectured by \cite{adler93}, who proved the
marginal convergence for any fixed $t$ and $x$.
\end{remark}
\begin{remark}The assumption that the offspring distribution is
Poisson with mean 1 can be relaxed. All that is really needed is that the
offspring distribution has an exponentially decaying tail.  See
Remark~\ref{rmk:thm1_general}.
\end{remark}

\begin{remark}\label{rmk:mu_smooth}
The hypothesis \eqref{ass:mu_conv} does not by itself imply
\eqref{ass:lt_smooth}, even if the limit measure $\mu $ satisfies
the hypothesis \eqref{ass:mu_smooth} of Sugitani's theorem.
Sufficient conditions for \eqref{ass:lt_smooth} are given in
Proposition \ref{prop:mu_smooth} below. In particular, in
dimension 2, if (\ref{ass:mu_conv}) holds and the maximal number
of particles on a single site is bounded in $k$, then
(\ref{ass:lt_smooth}) is satisfied.
\end{remark}

\begin{remark}\label{remark:support}
Let $X_{t}$ be super-Brownian motion in dimension $d=2$.  For each
$t>0$ the random measure $X_{t}$ is singular, so by Fubini's
theorem, for almost every point $x\in \zz{R}^{2}$ the set of times
$t>0$ such that $x$ is a point of density of $X_{t}$ has Lebesgue
measure $0$. Under hypothesis \eqref{ass:lt_smooth} we can make an
analogous quantitative statement for branching random walk: For
any fixed $x\in \zz{Z}^2$ and all $t>0,$
\begin{equation}
\label{eqn:occ_time}
   E\sum_{m=1}^{[kt]} I_{X^k_m(x)> 0} = O(k/\log k).
\end{equation}

\begin{proof}
By Proposition 35 in \cite{lz07}, there exists $\delta>0$ such
that for all $k$ and $m$ sufficiently large,
$$
   E\left[X^k_m(x) | X^k_m(x)>0\right] \geq \delta \log m.
$$
But hypothesis \eqref{ass:lt_smooth} implies that
$$
  ER_{kt}^k(x)=(\mu^kG_{kt})(x)=O(k),
$$
and
$$
  ER_{kt}^k(x)
  =\sum_{m< kt} E X^k_m(x)
  = \sum_{m<kt} E\left[X^k_m(x) | X^k_m(x)>0\right]\cdot P\left[X^k_m(x)>0\right].
$$
\end{proof}
\end{remark}

\subsection{Scaling limit of spatial \emph{SIR} epidemic}
 Before
stating our result, we first recall the definition of
Dawson-Watanabe processes with variable-rate killing. The
Dawson-Watanabe process $X_t$ with killing rate $\theta =
\theta(x, t,\omega)$ (assumed to be progressively measurable and
jointly continuous in $(t, x)$) and variance parameter $\sigma^2$
can be characterized by a martingale problem (\cite{DP99}, \S
6.2): For any test function $\psi\in C_c^2(\zz{R}^d)$,
$$
 \la X_t,\psi\ra - \la X_0,\psi\ra -\frac{\sigma d}{2}\int_0^t\la X_s,\Delta\psi\ra \,ds
  + \int_0^t\la X_s,\theta(\cdot,s)\psi\ra \,ds
$$
is a martingale with the same quadratic variation as for
super-Brownian motion with variance parameter $\sigma^{2}$. The
Dawson-Watanabe process with killing rate $0$ (which we sometimes
refer to as the \emph{standard Dawson-Watanabe process}) is
super-Brownian motion.  Existence and distributional uniqueness of
Dawson-Watanabe processes in general is asserted in \cite{DP99} and
proved, in various cases, in \cite{Dawson78} and \cite{EP95}. It is
also proved in these articles that the law of a Dawson- Watanabe
process with killing on a finite time interval is absolutely
continuous with respect to that of a standard Dawson-Watanabe process
with the same variance parameter, and that the likelihood ratio
(Radon-Nikodym derivative) is (\cite{DP99})
\begin{equation}\label{eq:RN}
    \exp\left\{-\int\theta(t,x)\,dM(t,x) -
                         \frac{1}{2}\int\la X_t,\theta(t,\cdot)^2\ra\, dt\right\},
\end{equation}
where $dM(t, x)$ is the {\it orthogonal martingale measure}
attached to the standard Dawson-Watanabe process (see
\cite{Walsh86}). Absolute continuity implies that sample path
properties are inherited: In particular, when $d=2, 3,$ if $X_t$
is a Dawson-Watanabe process with killing, then almost surely its
occupation time process $L_t$ is absolutely continuous, with local
time density $L_t(x)$ jointly continuous in $x$ and $t$.

It is shown in \cite{lalley07} that for the \emph{SIR}-1 epidemic in
$\zz{Z}$ with village size $N$, the particle density processes,
suitably rescaled, converge as $N\to \infty$ to the density process of
a standard Dawson-Watanabe process or a Dawson-Watanabe process with
location-dependent killing, depending on whether the total number of
initial infections is below a critical threshold or not. In dimensions
$d\geq 2$, one cannot expect such a result to hold, because the
Dawson-Watanabe process is almost surely singular with respect to the
Lebesgue measure and therefore has no associated density
process. However, as \emph{measure-valued} processes, the
\emph{SIR}-$d$ ($d=2,3$) epidemics, under suitable scaling, do
converge, as the next theorem asserts.  For the \emph{SIR}-$d$ model
with village size $N$, define
\begin{align}
\label{eq:SIR-X}
    X_i^N(x) :&=\#\text{ infected particles at }\; x \; \text{at
    time} \; i;\\
   \notag  R_n^N(x):&=\#\text{ recovered particles at }\; x \;
   \text{at time } n = \sum_{ i< n} X_i^N(x).
\end{align}

\begin{thm}\label{thm:SIR} Assume that $d=2$ or $3$, and suppose that for
some $\alpha \leq 1/(3-d/2)$ the initial configurations $\mu^N:=X_0^N$
are such that
\begin{align}
\label{hyp:initialConfig}
   & \mathcal{F}_{N^{\alpha}}\mu^{N} \Rightarrow  \mu    \text{ with compact support,\quad and}\\
\label{hyp:smoothness} &  ((\mu^NG_{N^{\alpha}t})
(\sqrt{N^{\alpha}}x))/N^{\alpha(2-d/2)}\Rightarrow [(q_{\sigma^2
t} *\mu)/\sigma^2](x)\in C(\zz{R}^{1+d})
\end{align}
 where the second convergence is in
$D([0,\infty);C_b(\zz{R}^d))$. Then
\begin{equation}\label{eq:SIR-convergence}
    (\mathcal{F}_{N^{\alpha}}X^{N})_{N^{\alpha}t}
    \Longrightarrow X_{t}
\end{equation}
where the limit process $X_{t}$ is a Dawson-Watanabe process with
initial configuration $X_{0}=\mu$, variance parameter $\sigma^{2}$,
and killing rate $\theta$.  The killing rate depends on the value of
$\alpha$ as follows:
\begin{itemize}
  \item[(i)] if $\alpha<1/(3-d/2)$, then $\theta\equiv 0$;\quad and
  \item[(ii)] if $\alpha=1/(3-d/2)$, then $\theta=L_t(x)$,
\end{itemize}
where $L_t(x)$ is the local time density of the process $X_{t}$. The
convergence $\Rightarrow$ in \eqref{eq:SIR-convergence} is weak
convergence relative to the Skorokhod topology on $D([0,\infty);M_F(\zz{R}^d))$.
\end{thm}

Theorem \ref{thm:SIR} will be proved in \S3.

\begin{remark}\label{remark:criticalThreshold}
Theorem~\ref{thm:SIR} asserts that there is a \emph{critical
threshold} for the \emph{SIR}$-d$ epidemic in dimensions $d=2,3$:
Below the threshold (when the sizes of the initially infected
populations are $\ll N^{\alpha_{*}}$, where $\alpha_{*}=1/(3-d/2)$ is
the \emph{critical exponent}) the effect of finite population size is
not felt, and the epidemic looks much like its branching envelope.  At
the critical threshold, the finite-population effects begin to show,
and the epidemic now looks like a branching random walk with
location-dependent killing.
\end{remark}

\begin{remark}
The critical behavior of the \emph{SIR}-$d$ epidemics in
dimensions $d\geq 4$ is considerably simpler.  It can be shown
that in all dimensions $d\geq 5$, the critical threshold is
$\alpha=1$; in the four dimensional case, the critical threshold
is still $\alpha=1$ but there will be logarithmic corrections --
we thank Ed Perkins for pointing this out to us. 
\end{remark}

\subsection{Notational conventions} Since the proof of
Theorem~\ref{thm:SIR} is based on likelihood ratio calculations,
we shall, at the risk of minor confusion, use the same letters $X$
and $R$, with subscripts and/or superscripts, to denote particle
counts and occupation counts for both branching random walks and
the \emph{SIR}-$d$ epidemic processes (see equations \eqref{eq:RN}
and \eqref{eq:SIR-X}) and for their continuous limits.  Throughout
the paper, we use the notation $f\asymp g$ to mean that the ratio
$f/g$ remains bounded away from $0$ and $\infty$. Also, $C,C_1,$
\hbox{etc.} denote generic constants whose values may change from
line to line. The notation $\delta_{x} (y)$ is reserved for the
Kronecker delta function. The notation $Y_{n}=o_{P} (f (n))$ means
that $Y_{n}/f (n) \rightarrow 0$ in probability; and $Y_{n}=O_{P}
(f (n))$ means that the sequence $|Y_{n}|/f (n)$ is tight.
Finally, we use a ``local scoping rule'' for notation: Any
notation introduced in a proof is local to the proof, unless
otherwise indicated.

\section{Local time for branching random walk in $d=2,3$}

\subsection{Estimates on transition probabilities}\label{ssec:localLimitLaws}
Recall that $\zz{P}^{n}= (P_{n} (x-y))$ is the $n-$step transition
probability kernel for the simple  random walk on $\zz{Z}^{d}$
(with holding parameter $1/ (2d+1)$). For critical branching
random walk, $P_{n} (x,y)$ is the expected number of particles at
site $y$ at time $n$ given that the branching random walk is
initiated by a single particle at site $x$. For this reason, sharp
estimates on these transition probabilities will be of crucial
importance in the proof of Theorem~\ref{thm:lt}. We collect
several useful estimates here. As the proofs are somewhat
technical, we relegate them to the Appendix (section~\ref{sec:app}
below).  Write
\begin{align*}
    \Phi_n(x,y)&=\phi_n(x)+\phi_n(y) \quad
    \text{where} \\
    \phi_n(x)&=\frac{1}{ (2\pi  n)^{d/2}}\exp\left(-\frac{|x|^{2}}{2n}\right)
\end{align*}
is the Gauss kernel in $\zz{R}^{d}$. The first two results relate
transition probabilities to the Gauss kernel.

\begin{lemma}\label{lemma:lltA}
For all sufficiently small $\beta >0$ there exists constant  $C=C
(\beta)>0$ such that  for all integers $m,n\geq 1$ and all $x\in
\zz{Z}^{d}$,
\begin{align}\label{eqn:lclt_bd}
     P_n(x)&\leq C\phi_n(\beta x) \quad \text{and}\\
\label{eqn:dis_conv}
   (P_m * \phi_n)(\beta x)&\leq C\phi_{m+n}(\beta x/2).
\end{align}
Furthermore, for each $A >0$ and each $T>0$ there exists $C=C(A,
T)>0$ such that for all~$k$ sufficiently large and all $|x|\leq A
\sqrt{k}$,
\begin{equation}\label{eqn:green_bd}
  \sum_{n\leq k T}\phi_n(\beta x) \leq C\sum_{n\leq kT}P_n(x).
\end{equation}
\end{lemma}

\begin{lemma}\label{lemma:lclt_diff}
For all sufficiently small $\beta >0$ there exists constant  $C=C
(\beta)>0$ such that for all integers $n\geq 1$ and all $x,y\in
\zz{Z}^{d}$,
\begin{equation}\label{eqn:p_diff}
     |P_n(x) - P_n(y)|\leq
      C \left(\frac{|x-y|}
      {\sqrt{n}}\wedge 1\right)\cdot\Phi_n(\beta x,\beta y).
\end{equation}
In particular, for all $ \gamma \leq 1$,
\begin{equation}\label{p_diff_alpha}
     |P_n(x) - P_n(y)|\leq C
     \left( \frac{|x-y|}{\sqrt{n}}\right)^{\gamma}
    \cdot\Phi_n(\beta x,\beta y).
\end{equation}
\end{lemma}

Our arguments will also require the following estimates on the
discretized Green kernel.

\begin{lemma}\label{lemma:greenA}
For each $\gamma\in(0,2-d/2)$, $\beta>0, h\geq 1,$ $n\in\zz{N}$,
and $x,y\in \zz{Z}^{d}$, define
\begin{equation}\label{eq:def_F}
   F_{n,h}(x,y;\beta)=F_{n,h;\gamma}(x,y;\beta)
   =\sum_{|\rho|<h}\sum_{l<n} \frac{1}{l^{\gamma/2}}
        \Phi_l(\beta(x+\rho),\beta(y+\rho)),
\end{equation}
where the first summation is over $\rho\in\zz{Z}^d$ with
$|\rho|<h$.
 Then there exists $C=C (\gamma,\beta,d)<\infty$
such that for all $n\in\zz{N}$, $h_1,h_2\geq 1$, and all $x,y\in
\zz{Z}^{d}$, the following inequalities hold:
\begin{equation}\label{eqn:green_indc}
   F_{n,h_1}(x,y;\beta)\cdot  F_{n,h_2}(x,y;\beta)
   \leq C n^{2-(d+\gamma)/2} F_{n,h_1+h_2-1}(x,y;\beta),
\end{equation}
and
\begin{equation}\label{eqn:conv}
\aligned
   &\sum_{i<n} \sum_{z} P_i(z)\cdot\left[F_{n-i,h_1}(x-z,y-z;\beta)\cdot  F_{n-i,h_2}(x-z,y-z;\beta)
   \right]\\
   &\leq C
   n^{2-(d+2\gamma)/2}\sum_{|\rho|<h_1+h_2-1}\sum_{l<n}\Phi_l(\beta(x+\rho)/2,\beta(y+\rho)/2)\\
   &\leq C n^{2-(d+\gamma)/2}F_{n,h_1+h_2-1}(x,y;\beta/2).
\endaligned
\end{equation}
(Note that in the last term the $\beta$ parameter is changed to $\beta/2$.)
\end{lemma}

\begin{lemma}\label{lemma:greenB}
For each $\beta>0, h\geq 1$, $m, n\in\zz{N}$,  and $x\in
\zz{Z}^{d}$, define
\begin{equation}\label{eq:def_J}
   J_{m,n,h}(x;\beta)=
   \sum_{|\rho|<h}\sum_{m\leq l<m+n} \phi_l(\beta(x+\rho)),
\end{equation}
where the first summation is over $\rho\in\zz{Z}^d$ with
$|\rho|<h$. Then there exists $C=C(\beta)>0$ such that for all $m,
n\in\zz{N}$, $h_1,h_2\geq 1$, and all $x\in \zz{Z}^{d}$, the
following inequalities holds:
\begin{equation}\label{eqn:green_indc_B}
   J_{m,n,h_1}(x;\beta)\cdot  J_{m,n,h_2}(x;\beta)
   \leq C n^{2-d/2} J_{m,n,h_1+h_2-1}(x;\beta),
\end{equation}
and
\begin{equation}\label{eqn:conv_B}
   \sum_{i<n} \sum_{z} P_i(z)\cdot\left[J_{m,n-i,h_1}(x-z;\beta)\cdot  J_{m,n-i,h_2}(x-z;\beta)
   \right]
   \leq C n^{2-d/2}J_{m,n,h_1+h_2-1}(x;\beta/2).
\end{equation}
\end{lemma}

\subsection{Proof of Theorem \ref{thm:lt}}\label{ssec:proofThLLT}
For notational ease, we omit the superscript $k$ in the arguments
below: thus, we write $X_{n} (x)$ instead of $X^{k}_{n} (x)$, and
$R_{n} (x)$ instead of $R^{k}_{n} (x)$. To prove the theorem it
suffices to prove that (1) the sequence of random processes
$(R_{kt}(\sqrt{k}x)/k^{2-d/2})$ is tight in the space
$D([0,\infty); C_b(\zz{R}^d))$; and (2) that the only possible
weak limit is the local time density process $L_{t} (x)$. The
second of these is easy, given Lemma~\ref{lemma:jtconv_sbm_lt}.
This implies that for any test function $ \psi\in C_c(\zz{R}^d)$,
$$ \aligned
  & \frac{1}{k^{2}}\sum_x R_{kt}(\sqrt{k}x)\psi(x) \\
  &=\frac{1}{k}\sum_{i\leq kt} \sum_x X_i(\sqrt{k} x)\psi(x)/k\\
  &\to_{\mathcal{L}} \int_0^t X_s(\psi)\,ds,
\endaligned $$
where $(X_t)$ is the super-Brownian motion started in
configuration $X_{0}=\mu$, run at speed~$\sigma$. Hence, any weak
limit of the sequence $(R_{kt}(\sqrt{k}x)/k^{2-d/2})$ must be a
density of the occupation measure for super-Brownian motion. On
the other hand, by Remark \ref{rmk:mu_smooth} and Sugitani's
Theorem,
$$
  \int_0^t X_s(\psi)\,ds = \int_x L_t(x)\psi(x)\, dx.
$$
It follows that $L_{t} (x)$ is the only possible weak limit.

Thus, to prove Theorem~\ref{thm:lt} it suffices to prove that the
sequence  $(R_{kt}(\sqrt{k}x)/k^{2-d/2})$ is tight
in the space $D([0,\infty); C_b(\zz{R}^d))$. In view of hypothesis
\eqref{ass:mu_smooth}, it is enough to prove the  tightness of the
re--centered sequence
\begin{equation}\label{eq:R-recentered}
    Y_{k} (t,x):=\left(R_{kt}(\sqrt{k}x)-(\mu^k G_{kt})(\sqrt{k}x)\right)\!/\,k^{2-d/2}.
\end{equation}
This we will accomplish by verifying a form of the
Kolmogorov-Centsov criterion. According to this criterion, to
prove tightness it suffices to prove that for each compact subset
$K$ of $[0,\infty )\times \zz{R}^{d}$ there exist constants
$C<\infty$, $\alpha >0$, and $\delta >d+1$ such that for all pairs
$(s,a), (t,b)\in K$,
\begin{align}\label{eq:kc-critSpace}
    E|Y_{k} (t,a)-Y_{k} (t,b)|^{\alpha }&\leq C |a-b|^{\delta } \quad \text{and}\\
\label{eq:kc-critTime}
    E|Y_{k} (t,a)-Y_{k} (s,a)|^{\alpha }&\leq C |t-s|^{\delta }     .
\end{align}
The trick is to not work with moments
directly, but instead, following the strategy of \cite{Sugitani89}, to
work with cumulants:

\begin{lemma}[Lemma 3.1 in \cite{Sugitani89}]\label{lemma:cum_mom}
 Let $X$ be a random variable with moment generating function
$E\exp(\theta X) =\exp(\sum_{n=1}^\infty\theta^n a_n)$. If for
some integer $N$ there exists $r,b>~0$ such that
$$
 |a_n|\leq b r^n,\text{ for } n\leq 2N,
$$
then there exists $ C= C(b,N)>0$ such that
$$
  E X^{2N}\leq Cr^{2N}.
$$
\end{lemma}

\bigskip \noindent \textbf{A. Cumulants.}
In the following discussion we use the notation $\ip{\nu,f}$ to
denote the inner product of a function $f$ and a measure $\nu$ on
$\zz{Z}^{d}$, and we let $R_{n}$ be the occupation measure of the
nearest neighbor branching random walk with Poisson(1) offspring
distribution. By the additivity and spatial homogeneity of the
branching random walk, for any $\psi\in C_c(\zz{Z}^d)$ and for
each $n\geq 1$ there exists a function $\nu_{n}=\nu_n^{\psi} \in
C_c(\zz{Z}^d)$ such that for any (nonrandom) initial configuration
$\mu$,
\[
    E^{\mu}\exp (\ip{R_{n}, \psi})=\exp (\ip{\mu ,\nu_{n}}).
\]
Note that $\nu_{1}=\psi$. The assignment $\psi \mapsto
\nu_{n}^{\psi}$ is monotone in $\psi$, but not in general linear.
Setting $\mu=\delta_x$ and conditioning on the first generation,
we obtain
\[
    \exp(\nu_{n+1}(x)) =\sum_j Q_j \left(\frac{1}{2d+1}\sum_e
    \exp(\psi (x+e)+\nu_n(x+e))\right)^{\!j}
\]
where $\{Q_{j} \}$ is the offspring distribution (in the case of
interest, the Poisson distribution with mean~$1$) and the inner
sum is over the $2d+1$ nearest neighbors $e$ of the origin in
$\zz{Z}^{d}$ (recall that the origin is included in this
collection, since particles of the branching random walk can stay
at the same sites as their parents). Observe that if the offspring
distribution is Poisson(1),  then
\begin{equation}
\label{eqn:nu_indc_s}
    \nu_{n+1}(x) = \frac{1}{2d+1}\sum_e \exp(\psi (x+e)+\nu_n(x+e)) -1.
\end{equation}

Define the \emph{cumulants}
$\kappa_{h,n}(x)=\kappa_{h,n}^\psi(x)$  in the usual way:
\[
    E^\mu \exp(\theta \la R_n, \psi\ra)
     =\exp\left(\left\la \mu, \sum_{h\geq 1}\theta^h \kappa_{h,n}\right\ra \right),\quad\forall\ \theta\in\zz{R}.
\]
By  the arguments of the preceding paragraph, $\kappa_{1,1} =\psi
$ and $\kappa_{h,1}=0$ for all $h\geq 2$, and
by~\eqref{eqn:nu_indc_s},
$$
    \sum_{h\geq 1}\theta^h \kappa_{h,n+1}(x)
     = \frac{1}{2d+1}\sum_e
       \left(
          \exp\left\{ \sum_{h\geq 1} \theta^h
          \kappa_{h,n}(x+e)+\theta\psi (x+e)\right\}  -1
     \right)
     .
$$
Consequently,
\begin{equation}\label{eqn:nuk_indc_0}
     \kappa_{h,n+1}(x) = \frac{1}{2d+1}\sum_{e} \sum_{m=1}^{h}
              \frac{1}{m!} \sum_{\mathcal{P}_{m} (h)}
                       \prod_{i=1}^{m}
                       \{\kappa_{h_{i},n} (x+e) + \delta_{1} (h_{i}) \cdot \psi (x+e)
                       \}\\
\end{equation}
where $\mathcal{P}_{m} (h)$ denotes the set of $m-$tuples
$(h_{1},h_{2},\dotsc ,h_{m})$ of positive integers whose sum is
$h$, and $\delta_{1} (\cdot)$ is the Kronecker delta function.
When $h\geq 2$, the $m=1$ summand in \eqref{eqn:nuk_indc_0} equals
$1/(2d+1)\cdot\sum_{e}\kappa_{h,n} (x+e)=(P_1 * \kappa_{h,n})(x)$,
hence,
$$
 \kappa_{h,n+1}(x)=(P_1 * \kappa_{h,n})(x)
                       + \Xi_{n+1}(x),
$$
where
\begin{equation}\label{eqn:xi}
  \Xi_{n+1}(x)=\Xi_{n+1}(x;h):= \frac{1}{2d+1}\sum_{e} \sum_{m=2}^{h}
              \frac{1}{m!} \sum_{\mathcal{P}_{m} (h)}
                       \prod_{i=1}^{m}
                       \{\kappa_{h_{i},n} (x+e) + \delta_{1} (h_{i}) \psi (x+e)
                       \}.
\end{equation}
Since $\kappa_{h,1}=0$ for all $h\geq 2$, by iteration we then get
\begin{equation}\label{eqn:nuk_indc}
   \kappa_{h,n}(x)=\sum_{l=0}^{n-1}(P_{l} * \Xi_{n-l})(x).
\end{equation}

Consider now the special case
$\psi=\psi_{a,b}:=\delta_{a}-\delta_{b}$ where $ a,b\in \zz{Z}^d$
and $\delta_{x}$ is the Kronecker delta function.  Fix $0< \gamma
<2-d/2$ small, and let
\[
  \eta=\eta(\gamma)=2-(d+\gamma)/2>0.
\]
 Recall that in \eqref{eq:def_F} in Lemma \ref{lemma:greenA} we
defined $F_{n,h}(x,y;\beta)$ for $\beta>0, h\geq 1,n\in\zz{N}$ and
$x,y\in\zz{Z}^d$ as
$$
     F_{n,h}(x,y;\beta)=F_{n,h;\gamma}(x,y;\beta)
   =\sum_{|\rho|<h}\sum_{l<n} \frac{1}{l^{\gamma/2}}
        \Phi_l(\beta(x+\rho),\beta(y+\rho)).
$$
\begin{claim}
For each $h\geq 1$ there exists $C_{h}<\infty$ such that for all
$n\in\zz{N}$ and all $x\in\zz{Z}^{d}$,
\begin{equation}\label{eqn:k_gen}
     |\kappa_{h,n}(x)|
     \leq    C_h|a-b|^{h\gamma}n^{\eta(h-1)}F_{n,h}(a-x,b-x;2^{-(h-1)}\beta);
\end{equation}
moreover, for all $ h\geq 2$, all $n\in\zz{N}$ and all
$x\in\zz{Z}^{d}$,
\begin{equation}\label{eqn:nu_est}
  |\kappa_{h,n}(x)|
  \leq C_h|a-b|^{h\gamma} n^{\eta(h-1)-\gamma/2}\sum_{|\rho|<h}\sum_{l<n}
   \Phi_l(2^{-(h-1)}\beta(a-x+\rho),2^{-(h-1)}\beta(b-x+\rho)).
\end{equation}
\end{claim}
In fact, when $h=1$,
$$\aligned
    |\kappa_{1,n}(x)|& = |E^{\delta_x}\la R_n,\psi\ra|
     =|G_n(a-x) -    G_n(b-x)|\\
    &\leq  C|a-b|^{\gamma }\sum_{l<n} \frac{1}{l^{ \gamma/2}} \Phi_l(\beta (a-x),\beta (b-x))\\
    &=C|a-b|^{\gamma }F_{n,1}(a-x,b-x;\beta),
\endaligned
$$
where in the middle inequality we used  assertion
(\ref{p_diff_alpha}) of Lemma \ref{lemma:lclt_diff}. Furthermore,
since $\psi(x)\neq 0$ if and only if $x=a$ or $b$, in which case
$|\psi(x)|=1$ and \\ $\inf_n C|a-b|^{\gamma
}F_{n,1}(a-x,b-x;\beta)
>0$, we get that for all $n$ and all $x$,
\begin{equation}\label{eq:k_1}
   |\kappa_{1,n}(x)+\psi(x)|\leq C |a-b|^{\gamma }F_{n,1}(a-x,b-x;\beta).
\end{equation}
Now suppose that the claim holds for $1,\ldots,h-1$, and we want
to prove the claim for $h$. First note that in the definition
\eqref{eqn:xi} of $\Xi_n(x)$,  only $\kappa_{h_i}$ for $h_i<h$ are
involved, hence by induction, \eqref{eq:k_1} and
relation~\eqref{eqn:green_indc} we get that for all $n$ and $x$,
$$\aligned
  \Xi_{n}(x)&\leq C |a-b|^{h\gamma} n^{\eta(h-2)}\sum_{e} \sum_{m=2}^{h}
             \sum_{\mathcal{P}_{m} (h)}\\
             &\quad \cdot F_{n,h_1}(a-x+e,b-x+e;2^{-(h_1-1)}\beta)
              \cdot F_{n,h-h_1}(a-x+e,b-x+e;2^{-(h-2)}\beta).
\endaligned
$$
The claims then follows from \eqref{eqn:nuk_indc} and
\eqref{eqn:conv}.

\bigskip \noindent
\textbf{B. Proof of \eqref{eq:kc-critSpace}.} Suppose the initial
configurations $\mu^k$ satisfy the hypotheses of
Theorem~\ref{thm:lt}. For any $a, b\in \zz{R}^d$, we want to
estimate $E^{\mu^k}|Y_{k} (t,a)-Y_{k} (t,b) |^m$. By Lemma
\ref{lemma:cum_mom}, this can be done by setting $\psi=1_{\sqrt{k}
a} - 1_{\sqrt{k} b}$ and estimating $|\la \mu^k,
\kappa_{h,kt}\ra|$.  By (\ref{eqn:nu_est}), for all $ h\geq 2$,
$$
   \aligned
   & |\la \mu^k, \kappa_{h,kt}\ra|
   \leq C_h
   |\sqrt{k}(a-b)|^{h\gamma}\cdot(kt)^{\eta(h-1)-\gamma/2}\cdot\\
   &\hskip 3cm\cdot\sum_{|\rho|<h}\la\mu^k,\sum_{l<kt}
   \Phi_l(2^{-(h-1)}\beta(\sqrt{k}a+\rho-\cdot),2^{-(h-1)}\beta(\sqrt{k}b+\rho-\cdot))\ra.
   \endaligned
$$
 By assumption (\ref{ass:mu_conv}),  we can find an $A>0$ such
that $\text{supp}(\mu^k)\subseteq B(0,A\sqrt{k})$ for all $k,$
where $B(0,r)$ represents the ball of radius $r$ around 0. We
therefore have that for all $k$ sufficiently large,
\begin{equation}\label{eqn:muG_bdd}
   \aligned \max_x \la\mu^k, \sum_{l< k t}\phi_l(2^{-(h-1)}\beta(x-\cdot))\ra
   &=\max_{|x|\leq A\sqrt{k}} \sum_y \mu_k(y) \sum_{l< k t} \phi_l(2^{-(h-1)}\beta(x-y))\\
   &\leq C\sum_y \mu_k(y)\sum_{l< k t}P_l(x-y)\\
   &\leq C k^{2-d/2},
   \endaligned
\end{equation}
where in the first inequality we used (\ref{eqn:green_bd}), and
the last inequality is due to the relative compactness of
$(\mu^kG_{kt})(\sqrt{k}\cdot)/k^{2-d/2}$ assumed in
(\ref{ass:lt_smooth}).
 Hence when $h\geq 2$,
$$
   |\la \mu^k, \kappa_{h,kt}\ra|
   \leq C k^{h\gamma/2 + \eta(h-1)-\gamma/2 +
    2-d/2}\cdot  t^{\eta(h-1)-\gamma/2}\cdot
    |a-b|^{h\gamma}.
$$
Plugging in $\eta = 2-(d+\gamma)/2$ gives us
$$
   |\la \mu^k, \kappa_{h,kt}\ra|
   \leq C k^{(2-d/2)h} \cdot t^{(2-(d+\gamma))(h-1)-\gamma/2}\cdot
    |a-b|^{h\gamma}.
$$
Noting that $E^{\mu^k}\left(Y_{k} (t,a)-Y_{k} (t,b)\right)=0$, by
Lemma \ref{lemma:cum_mom}, we get
$$
   E^{\mu^k}\left|Y_{k} (t,a)-Y_{k} (t,b)\right|^{2h}
   \leq C t^{(2-(d+\gamma)/2)(2h-1)-\gamma/2}\cdot |a-b|^{2h\gamma}.
$$
By choosing $h$ large such that
 $2h\gamma> d+1$ we obtain \eqref{eq:kc-critSpace}.

\bigskip \noindent
\textbf{C. Proof of \eqref{eq:kc-critTime}.}
By the additivity and
spatial homogeneity of the branching random walk, for any $\psi\in
C_c(\zz{Z}^d)$ and for all  $ m,n\in \zz{N}$ there exists a
function $\nu_{n}=\nu_n^\psi \in C_c(\zz{Z}^d)$ such that for any
(nonrandom) initial configuration $\mu$,
$$
 E^\mu \exp(\langle R_{n+m}-R_m, \psi\rangle)
 = \exp(\la\mu, \nu_{(n,m)}\ra).
$$
Letting $\mu=\delta_x$ and conditioning on the
first generation, we get
$$\aligned
  \exp(\nu_{(n,m+1)}(x))
  &= E^{\delta_x} \exp(\la R_{n+1+m} - R_{m+1}, \phi\ra)\\
  &=\sum_j Q_j \left(\frac{1}{2d+1}\sum_e \exp\left( \nu_{(n,m)}(x+e)\right)\right)^{\!j},
\endaligned$$
where $\mathcal{Q}=\{Q_j\}_{j\geq 0}$ denotes the offspring
distribution. In case where the offspring distribution is
Poisson(1), the  equation above implies
\begin{equation}   \label{eqn:nu_indc}
  \nu_{(n,m+1)}(x) = \frac{1}{2d+1}\sum_e \exp\left(
  \nu_{(n,m)}(x+e)\right) -1.
\end{equation}
Define the cumulants $\kappa_{h,(n,m)}$ by
$$
  E^\mu   \exp(\theta\la R_{n+m}-R_m, \psi\ra)
  =\exp(\la \mu,   \sum_{h\geq 1}\theta^h
  \kappa_{h,(n,m)}\ra),\quad\forall\ \theta\in\zz{R}.
$$
Then by (\ref{eqn:nu_indc}),
$$
  \sum_h \theta^h \kappa_{h,(n,m+1)}(x)
  = \frac{1}{2d+1}\sum_e \exp\left[ \sum_h \theta^h \kappa_{h,(n,m)}(x+e)\right] -1.
$$
Therefore
\begin{equation}\label{eqn:nuk_indc_t}
  \kappa_{h,(n,m+1)}(x)
   = \frac{1}{2d+1}\sum_e \sum_{i=1}^h \frac{1}{i!}
  \sum_{\mathcal{P}_{i} (h) } \prod_{j=1}^i
  \kappa_{h_i,(n,m)}(x+e),
\end{equation}
where $\mathcal{P}_{i} (h)$ denotes the set of $i-$tuples
$(h_{1},h_{2},\dotsc ,h_{i})$ of positive integers whose sum is
$h$.  The $m=1$ summand in \eqref{eqn:nuk_indc_t} equals
$1/(2d+1)\cdot\sum_{e}\kappa_{h,(n,m)} (x+e)=(P_1*
\kappa_{h,(n,m)})(x)$, hence when $h\geq 2$,
$$
 \kappa_{h,(n,m+1)}(x)=(P_1* \kappa_{h,(n,m)})(x)
                       + \tilde{\Xi}_{n,m+1}(x),
$$
where
\begin{equation}\label{eqn:txi}
  \tilde{\Xi}_{n,m+1}(x)=\tilde{\Xi}_{n,m+1}(x;h):= \frac{1}{2d+1}\sum_{e} \sum_{i=2}^{h}
              \frac{1}{i!} \sum_{\mathcal{P}_{i} (h)}
               \prod_{j=1}^i \kappa_{h_i,(n,m)}(x+e)
\end{equation}
By iteration we then get that for all $h\geq 2$,
\begin{equation}\label{eqn:nuk_indc_2}
   \kappa_{h,(n,m)}(x)=\sum_{i<m}(P_i * \tilde{\Xi}_{n,m-i})(x).
\end{equation}
For  $\psi= 1_a$ for  $a\in \zz{Z}^d,$ by (\ref{eqn:lclt_bd})
$$
  \kappa_{1,(n,m)}(x) = E^{\delta_x}\la R_{n+m}-R_m, \psi\ra
  =\sum_{m\leq l<m+n} P_l(a-x)
  \leq C\sum_{m\leq l<m+n}   \phi_l(\beta(a-x)).
$$
Furthermore, similarly as in proving the claim in Part A, using
Lemma \ref{lemma:greenB} and (\ref{eqn:nuk_indc_2}) we get that
for all $h$, there exists $C_h>0$ such that
$$
  |\kappa_{h,(n,m)}(x)|\leq C_h n^{(2-d/2)(h-1)}\sum_{|\rho|<h}\sum_{m\leq l<m+n}
    \phi_l\left(2^{-(h-1)}\beta(a-x+\rho)\right) .
$$

We are ready to verify \eqref{eq:kc-critTime}. Setting $\psi=
1_{\sqrt{k} a}$ and using (\ref{eqn:muG_bdd}), we get
$$
 \aligned
 |\la \mu^k, \kappa_{h,(kt,ks)}\ra|
 & \leq C k^{(2-d/2)(h-1)} t^{(2-d/2)(h-1)}\cdot k^{2-d/2} \\
 &= C k^{h(2-d/2)} \cdot t^{(2-d/2)(h-1)}.
 \endaligned
$$
By Lemma \ref{lemma:cum_mom} we then get
$$
  E^{\mu^k}\left|Y_{k} (t+s,a)-Y_{k} (s,a) \right|^{2h} \leq C \cdot t^{(2-d/2)(2h-1)}.
$$
So by choosing $h$ large such that $(2-d/2)(2h-1)>1+d$ we obtain
\eqref{eq:kc-critTime}.

 \qed

\begin{remark}\label{rmk:thm1_general}
 For general offspring distributions $\mathcal{Q}=\{Q_j\}$,
let $f(x)=\log(\sum_j Q_j x^j)$ where $x\geq 0$. If the offspring
distribution $\mathcal{Q}$ has an exponentially decaying tail,
then $f(x)$ can be expanded around $x=1$ as
$f(x)=\sum_{\ell=1}^\infty f^{(\ell)}(1) (x-1)^\ell/\ell!$. Thus
(\ref{eqn:nu_indc_s}) turns into
$$
  \nu_{n+1}(x)
  =\sum_{\ell=1}^\infty f^{(\ell)}(1) \left( \frac{1}{2d+1}\sum_e \exp\left(\psi(x+e)
  + \nu_n(x+e)\right) -1\right)^{\!\ell}\!\!\bigg/\,\ell!,
$$
and
$$
  \sum_h \theta^h \kappa_{h,n+1}(x)
  =\sum_{\ell=1}^\infty f^{(\ell)}(1) \left( \frac{1}{2d+1}\sum_e \exp\left(\theta\psi(x+e) +
   \sum_h \theta^h \kappa_{h,n}(x+e)\right)-1\right)^{\!\ell}\!\!\bigg/\,\ell!.
$$
This enables us to express $\kappa_{h,n+1}(x)$ in terms of
$\psi(x+e)$ and $\kappa_{h,n}(x+e) $ similarly as in
(\ref{eqn:nuk_indc_0}) and in~\eqref{eqn:nuk_indc} (note
$f^{(1)}(1)=1$ because $\mathcal{Q}$ has mean 1), and prove the
Kolmogorov-Centsov criterion for the spatial variable. Similarly
one can verify the Kolmogorov-Centsov criterion for the time
variable.
\end{remark}

\subsection{Sufficient conditions for Assumption
(\ref{ass:lt_smooth})}\label{ssec:sc_ltsmooth} Now we state some
conditions that imply (\ref{ass:lt_smooth}) and are easier to check.

\begin{prop}\label{prop:mu_smooth} Let $d=2$ or $3$. Suppose that
the initial configurations $\mu^k$ are such that
$\mathcal{F}_{k}\mu^{k} \Rightarrow\mu$, and satisfy
\begin{equation}\label{ass:ini_spread}
   \lim_{t\to 0}\sup_k \max_x(\mu^kG_{kt})(x)/k^{2-d/2} =0.
\end{equation}
Then (\ref{ass:lt_smooth}) holds.
 In particular, if any of the following assumptions is satisfied,
 then (\ref{ass:lt_smooth}) holds.

\medskip\noindent
(i) In dimension 2, the maximal number of particles on a single
site is bounded in $k$, \hbox{i.e.},\\
  $\sup_k \max_y \mu^k(y)<\infty.$

\medskip\noindent
(ii)
 In dimension 3, there exist $C_1, C_2>0$ such that
\begin{equation}\label{ass:ini_spread_3d}
  C_2:=\sup_k \max_{x}  \sum_{y\in B(x,  3C_1 k^{1/6})} \mu^k(y)<\infty,
\end{equation}
where $B(x,r)$ denotes the ball of radius $r$ around $x$ for any
$x$ and $r\geq 0$;  that is, the number of particles in any ball
of radius $3C_1 k^{1/6}$ is bounded in $k$.

\medskip\noindent
(iii) In dimension 2, $\mu^k$ is such that $\mu^k(y)$ is a
decreasing function in $|y|$,
  and there exists  $\alpha\in~(0,2) $ such that
\begin{equation}\label{ass:ini_spread_2d_spike}
  \mu^k(y)  \leq C\left(\sqrt{k/(|y|^2+1)}\right)^{\!\alpha},\enspace \forall\, y, k.
\end{equation}

\end{prop}
\begin{remark}
This proposition is a natural analogue  of Proposition 1 in
\cite{Sugitani89}.
\end{remark}
To prove Proposition \ref{prop:mu_smooth}, we will need the
following result.

\begin{lemma}\label{lemma:green_conv} For any
function $\psi \in C_c(\zz{R}^d)$ and each integer $k\geq 1$, define
$$
 \Psi^k_t(x)
 =\sum_{y\in \zz{Z}^d} \psi(y/\sqrt{k}) G_{kt}(\sqrt{k}x -
y)/k,\quad \text{ for } \; x \in\zz{Z}^d/\sqrt{k} \; \text{and} \; t\in \zz{Z}/k,
$$
 and extend by linear interpolation elsewhere. Then
\begin{equation}\label{eq:Psi_{k}}
    \lim_{k \rightarrow \infty} \Psi^k_t(x) =[(q_{\sigma^2 t}
     *\psi)/\sigma^2](x),
\end{equation}
and the convergence is locally uniform in $t$ and $x$.
\end{lemma}

\begin{proof}
Pointwise convergence \eqref{eq:Psi_{k}} follows from the local
central limit theorem. To prove that the convergence is locally
uniform, it suffices to show that the sequence of
functions$\left(\Psi^k_t(x)\right)$ is relatively compact in
$C(\zz{R}^{1+d})$.  For this, we use the Ascoli-Arzela criterion.
First, we show that the functions $\Psi^k_t(x)$ are uniformly
bounded on any compact set in $\zz{R}^{1+d}$. Denote by $M$ the
maximum of $|\psi(x)|$. Then
\begin{equation} \label{eqn:psi_bdd}
 \aligned
 |\Psi^k_t(x)|
 &\leq \sum_{y\in \zz{Z}^d} |\psi(y/\sqrt{k})| \cdot G_{kt}(\sqrt{k}x
 - y)/k\\
 &\leq M \sum_{y\in \zz{Z}^d}  G_{kt}(\sqrt{k}x -
 y)/k\\
 &\leq Mt.
 \endaligned
\end{equation}
Next, we show they are equi-continuous. Fix $\varepsilon >0$, and
set $\delta =\varepsilon /M$.  By \eqref{eqn:psi_bdd},
$|\Psi^k_t|\leq \varepsilon$ for all $t\leq \delta$; thus,
\[
    |\Psi^k_t(x)-\Psi^k_s(y)|\leq 2\varepsilon,
        \quad \forall \; x,y\in \zz{R}^{d} \;\text{and} \; s,t\leq \delta.
\]
On the other hand, by (\ref{eqn:p_diff}), for all $t\geq\delta$ and  $x\neq y\in
\zz{R}^2$,
$$
 \aligned
 |\Psi^k_t(x)-\Psi^k_t(y)|
 &\leq 2\varepsilon + \frac{\sqrt{k}|x-y|}{k}\sum_{k\delta\leq n\leq kt}\frac{1}{\sqrt{n}}
  \sum_z |\psi(z/\sqrt{k})| \cdot \Phi_n(\beta(\sqrt{k}x -
    z),\beta(\sqrt{k}y - z))\\
 &\leq 2\varepsilon + \frac{\sqrt{k}|x-y|}{k}\sum_{k\delta\leq n\leq
 kt}\frac{1}{\sqrt{n}^{1+d}} \cdot C\sqrt{k}^d\\
 &\leq 2\varepsilon + C \delta^{-(d-1)/2}\cdot |x-y|.
 \endaligned
$$
(In the second inequality we used the fact that $\sum_z
|\psi(z/\sqrt{k})|\leq C\sqrt{k}^d$; this holds because $\psi$ is
bounded and has compact support.)  Finally, for all $x$ and all
$\delta\leq s< t$,
$$\aligned
  &|\Psi^k_t(x)-\Psi^k_s(x)|\\
   &\leq M \sum_{ks\leq n\leq kt} \sum_z
    P_n(\sqrt{k}x-z)/k\\
   &\leq M (t-s).
   \endaligned
$$
\end{proof}

\begin{proof}[Proof of Proposition \ref{prop:mu_smooth}]
For any $\psi\in C_c(\zz{R}^d)$, by Lemma \ref{lemma:green_conv},
$\Psi^k_t(x)$ converge to $[(q_{\sigma^2 t} *\psi)/\sigma^2](x)$
in the local uniform topology. Therefore,
$$
  \aligned \sum_x (\mu^kG_{kt})(\sqrt{k}x)/k^2\cdot \psi(x)
  &=\frac{1}{k}\sum_y \mu^k(\sqrt{k}y)\cdot \Psi^k_t(y)\\
  &\to \la \mu, [(q_{\sigma^2 t} *\psi)/\sigma^2]\ra
   \qquad (\because \mu^k(\sqrt{k}\cdot)/k \Rightarrow \mu\in M_F(\zz{R}^d))\\
  &=\la [(q_{\sigma^2 t} *\mu)/\sigma^2], \psi\ra.
  \endaligned
$$
On the other hand, if we can show that
$(\mu^kG_{kt})(\sqrt{k}x)/k^{2-d/2}$ is relatively compact in
$C(\zz{R}^{1+d})$, then for any limit $F(t,x)$,
$$
  \sum_x (\mu^kG_{kt})(\sqrt{k}x)/k^2 \psi(x)
  = \sum_x (\mu^kG_{kt})(\sqrt{k}x)/k^{2-d/2}\cdot \psi(x)
  \cdot 1/\sqrt{k}^d \to \int_x F(t,x) \psi(x)dx.
$$
Hence, $\la [(q_{\sigma^2 t} *\mu)/\sigma^2], \psi\ra = \int_x
F(t,x) \psi(x)dx$, which implies that (1) the measure
$[(q_{\sigma^2 t} *\mu)/\sigma^2]$ has density, and (2)
$(\mu^kG_{kt})(\sqrt{k}x)/k^{2-d/2}$ converge to $[(q_{\sigma^2 t}
*~\mu)/\sigma^2](x) $ in $C(\zz{R}^{1+d})$.

Now we show that \eqref{ass:ini_spread} implies that
$(\mu^kG_{kt})(\sqrt{k}x)/k^{2-d/2}$ is relatively compact in
$C(\zz{R}^{1+d})$, by verifying the Ascoli-Arzela criterion. We
first show that they are uniformly bounded on any compact set in
$\zz{R}^{1+d}$. In fact, by (\ref{ass:ini_spread}), there exists
$\delta>0$ such that
$$\sup_k\max_x (\mu^kG_{k\delta})(\sqrt{k}x)/k^{2-d/2} \leq 1; $$
moreover, for all $t\geq \delta $ and  all $x,$
$$
  \aligned
   & (\mu^kG_{kt})(\sqrt{k}x)/k^{2-d/2}\\
  &\leq 1 + \sum_{k\delta\leq n\leq kt} \sum_z \mu^k(z)P_n(\sqrt{k}x -z)/k^{2-d/2}\\
  &\leq 1 + \sum_{k\delta\leq n\leq kt}  C/n^{d/2} \cdot k /k^{2-d/2}\\
  &\leq C=C(t),
  \endaligned
$$
where in the second inequality we used the facts that there exists
$C>0$ such that for all $n$ and all  $x\in\zz{Z}^d$, $P_n(x)\leq
C/n^{d/2}$ (\hbox{cf.} \cite{spitzer76}, Proposition 6 on p72),
and that the total number of particles $\sum_z \mu^k(z)=O(k)$.

 Next we show they are equi-continuous. In fact, for any
$\varepsilon >0,$ by (\ref{ass:ini_spread}), there exists $\delta
>0$ such that
$$\sup_k\max_x (\mu^kG_{k\delta})(\sqrt{k}x)/k^{2-d/2} \leq \varepsilon; $$
therefore, for all $s,t\leq \delta $ and all $x,y,$
$$
  \sup_k |(\mu^kG_{ks})(\sqrt{k}x)/k^{2-d/2}- (\mu^kG_{kt})(\sqrt{k}y)/k^{2-d/2}| \leq 2\varepsilon;
$$
 moreover, for all $t\geq \delta$ and all $x\neq y$, by (\ref{eqn:p_diff}),
$$
  \aligned &|(\mu^kG_{kt})(\sqrt{k}x)/k^{2-d/2}- (\mu^kG_{kt})(\sqrt{k}y)/k^{2-d/2}|\\
   &\leq 2\varepsilon + \frac{1}{k^{2-d/2}}\sum_{k\delta\leq n\leq
   kt} \sum_z \mu^k(z)
  |P_n(\sqrt{k}x-z)-P_n(\sqrt{k}y-z) |\\
   &\leq 2\varepsilon + \frac{C\sqrt{k}|x-y|}{k^{2-d/2}}\sum_{k\delta\leq n\leq kt} \sum_z
     \mu^k(z) \frac{1}{\sqrt{n}} \Phi_n(\beta(\sqrt{k}x -
    z),\beta(\sqrt{k}y - z))\\
   &\leq 2\varepsilon + \frac{C\sqrt{k}|x-y|}{k^{2-d/2}}\sum_{k\delta \leq n\leq kt} \frac{1}{\sqrt{n}^{1+d}}\cdot k\\
   &\leq 2\varepsilon + C\delta^{-(d-1)/2}|x-y|;
  \endaligned
$$
and for all $x$ and all $\delta\leq s< t$,
$$\aligned
  &|(\mu^kG_{kt})(\sqrt{k}x)/k^{2-d/2}-
   (\mu^kG_{ks})(\sqrt{k}x)/k^{2-d/2}|\\
   &= \frac{1}{k^{2-d/2}}\sum_{ks\leq n\leq kt} \sum_z \mu^k(z)
    P_n(\sqrt{k}x-z)\\
   &\leq \frac{C}{k^{2-d/2}}\sum_{ks\leq n\leq kt} k/n^{d/2} \\
   &\leq \left\{\aligned &C\log (t/s)\leq C(t-s)/\delta, &\text{ if } d=2; \\
   &C(1/\sqrt{s} - 1/\sqrt{t})\leq C (t-s)/\delta^{3/2}, &\text { if } d=3. \endaligned          \right.
\endaligned
$$

We have therefore proved that (\ref{ass:ini_spread}) implies the
relative compactness of $(\mu^kG_{kt})(\sqrt{k}x)/k^{2-d/2}$. Next
we show that any of the conditions in (i)$\sim$(iii) implies
(\ref{ass:ini_spread}).

\medskip\noindent
(i) For all $x\in \zz{R}^2$ and all $ t\geq 0$,
$$\aligned
   (\mu^kG_{kt})(x)/k
   &= \frac{1}{k}\sum_{ n\leq kt} \sum_z \mu^k(z)P_n(x-z)\\
   &\leq \frac{C}{k}\sum_{n\leq kt} 1 =Ct,
\endaligned
$$
therefore (\ref{ass:ini_spread}) holds.

\medskip\noindent
(ii) In order to verify (\ref{ass:ini_spread}),  by
(\ref{eqn:lclt_bd}), it suffices to show that
\begin{equation}\label{ass:ini_3d}
  \lim_{t\to 0}\sup_k \max_x \sum_{n\leq kt}\sum_y
   \mu^k(y) \phi_n(\beta(x-y)) / \sqrt{k} = 0.
\end{equation}

\begin{claim}There exists $ C_3>0$ such that for all $k,n$ and all
pairs $x,y\in\zz{Z}^d$ with $|x-y|\geq C_1   k^{1/6}$,
$$
   \sqrt{k}\phi_n(\beta(x-y))\leq C_3\sum_{z\in B(y, C_1 k^{1/6})}
   \phi_n(\beta(x-z)).
$$
\end{claim}
In fact, the above inequality is equivalent to
$$
  \sqrt{k}\leq C_3 \sum_{|z-y|\leq C_1 k^{1/6}}
  \exp\left(\frac{\beta^2 (|x-y|^2-|x-z|^2)}{n}\right),\, \forall\, |x-y|\geq C_1   k^{1/6}.
$$
But this holds trivially since when $x\notin B(y, C_1 k^{1/6})$,
there is a positive proportion of integer points~$z$ in the ball
$B(y, C_1 k^{1/6})$  such that $|x-y|\geq |x-z|,$ and the
proportion does not depend on $k, x, y.$ Now let us estimate
$\sum_{n\leq kt}\sum_y \mu^k(y) \phi_n(\beta(x-y))/\sqrt{k}$: For
any fixed $k$ and $x$, this sum can be written as the sum of the
following two terms:
$$
   I:=\sum_{n\leq kt}\sum_{|y-x|\leq  C_1 k^{1/6}} \mu^k(y)\phi_n(\beta(x-y)) /
   \sqrt{k},
$$
and
$$
  II:=\sum_{n\leq kt}\sum_{|y-x|>  C_1 k^{1/6}} \mu^k(y)\phi_n(\beta(x-y)) /
  \sqrt{k}.
$$
As to term $I$, we have
$$ \aligned
  I
  &\leq \sum_{n\leq kt} \sum_{|y-x|\leq  C_1 k^{1/6}} \mu^k(y)\cdot
  C/n^{3/2}/\sqrt{k}\\
  &\leq \sum_{n\leq kt} C_2 \cdot
  C/n^{3/2}/\sqrt{k}\\
  &\leq C/\sqrt{k},
  \endaligned.
$$
where in the second inequality we used (\ref{ass:ini_spread_3d}).
 And by the claim and (\ref{ass:ini_spread_3d}),
$$
  \aligned
  II & \leq \sum_{n\leq kt}\sum_{|y-x|>  C_1 k^{1/6}} \mu^k(y)/
  \sqrt{k}\cdot C_3\sum_{z\in B(y, C_1 k^{1/6})}
   \phi_n(\beta(x-z))/\sqrt{k}\\
   &\leq C\sum_{n\leq kt}\sum_z \phi_n(\beta(x-z))\cdot \sum_{y\in B(z, C_1
   k^{1/6})} \mu^k(y)/
  k\\
  &\leq  \sum_{n\leq kt}C/ k\\
  &\leq Ct. \endaligned
$$
Therefore (\ref{ass:ini_3d}) holds.

\medskip\noindent
(iii) In order to verify (\ref{ass:ini_spread}), by
(\ref{eqn:lclt_bd}), it suffices to show that
\begin{equation}\label{ass:ini_2d}
  \lim_{t\to 0}\sup_k \max_x \sum_{n\leq kt}\sum_y \mu^k(y) \phi_n(\beta(x-y)) / k = 0.
\end{equation}
By assumption, $\mu^k(y)$ is a decreasing function of $|y|$; so is
$\phi_n(\beta y)$. Therefore, by Lemma \ref{lemma:fg_central}
below, the last term is bounded by $\sum_{n\leq kt} \sum_y
\mu^k(y) \phi_n(\beta y)/k$, which, by assumption
(\ref{ass:ini_spread_2d_spike}), can be further bounded by
$$\aligned
  &C\sum_{n\leq kt} \frac{1}{n\cdot k} \sum_y \left(\sqrt{\frac{k}{|y|^2+1}}\right)^\alpha
   e^{\frac{-\beta^2 |y|^2/k}{2n/k}} \\
   &\leq C\sum_{n\leq kt} \frac{1}{n\cdot k}\cdot k
     \int_{x\in \zz{R}^2} |x|^{-\alpha} e^{\frac{-\beta^2 |x|^2}{2n/k}}dx\\
   &\leq C \sum_{n\leq kt} \frac{1}{n}
     (n/k)^{-\alpha/2 + 1} \int_{x\in \zz{R}^2} |x|^{-\alpha} e^{-\beta^2 |x|^2/2}dx \\
   &\leq C \frac{1}{k}\cdot \sum_{n\leq kt} \left(\frac{n}{k}\right)^{\!-\alpha/2}\\
   &\leq C \int_0^t s^{-\alpha/2}\,ds\\
   &=O(t),
\endaligned$$
where the third inequality and the last equation hold because
$\alpha <2$ by assumption.
\end{proof}

\begin{remark}
In dimension 2, if the assumption in (iii) is satisfied, then the
radius of the support of $\mu^k$ will be of order $\sqrt{k}$. This
is because we need $\sum_y \mu^k(y)=O(k)$, hence for some $C>0$,
$$
  \sum_{y \in\text{Supp}(\mu^k)}\left(\sqrt{k/(|y|^2+1)}\right)^\alpha
  \geq C k,
$$
\hbox{i.e.},
$$
   \sum_{y \in \text{Supp}(\mu^k)} \left(1/\sqrt{|y|^2+1}\right)^\alpha\geq C k^{1-\alpha/2}.
$$
But for any $r,$
$$\aligned
  \sum_{|y|\leq r }\left(1/\sqrt{|y|^2+1}\right)^\alpha
  &=O\left(\int_{|y|\leq r}\left(1/\sqrt{|y|^2+1}\right)^\alpha \,dy
    \right)\\
  &=O\left(\int_0^r \left(1/\sqrt{s^2+1}\right)^\alpha\cdot s\,
     ds\right)\\
  &=O(r^{2-\alpha}),
\endaligned$$
so in order that $O(r^{2-\alpha})\geq C k^{1-\alpha/2}$, we need
$r=O(k^{1/2})$.
\end{remark}

\begin{remark}
In dimension 3, if $\mu^k$ is such that $\mu^k(y)$ is a decreasing
function in $|y|$,
  and there exists  $\alpha\in(0,2) $ such that
\begin{equation}\label{ass:ini_conc_3d}
   \mu^k(y) \leq C\left(\sqrt{k/(|y|^2+1)}\right)^\alpha\!/\sqrt{k},\enspace
     \forall\, y, k,
\end{equation}
 then using the similar proof as in (iii) we can show that
 (\ref{ass:ini_spread}) holds. But in fact,
 (\ref{ass:ini_conc_3d}) can never be satisfied. The reason is
 that in order $\mu^k(y)\geq 1$, we must have that for some $C>0$,
$$
    1\leq C \left(\sqrt{k/(|y|^2+1)}\right)^\alpha/\sqrt{k},
$$
which implies that
$$
   |y|\leq C k^{1/2-1/(2\alpha)}=o(k^{1/4}).
$$
On the other hand, we need $\sum_y \mu^k(y)=O(k)$, therefore for
some $C>0$,
$$
   \sum_{|y|=o(k^{1/4})}\left(\sqrt{k/(|y|^2+1)}\right)^\alpha/\sqrt{k}\geq C k,
$$
or
\begin{equation}\label{eqn:3d_im}
   \sum_{|y|=o(k^{1/4})} \left(1/\sqrt{|y|^2+1}\right)^\alpha \geq
   Ck^{(3-\alpha)/2}.
\end{equation}
However,
$$\aligned
   \sum_{|y|=o(k^{1/4})}\left(1/\sqrt{|y|^2+1}\right)^\alpha
   &=O\left(\int_{|y|=o(k^{1/4})} \left(1/\sqrt{|y|^2+1}\right)^\alpha \,dy\right)\\
   &=O\left(\int_0^{o(k^{1/4})} \left(1/\sqrt{r^2+1}\right)^\alpha
          \cdot r^2\, dr \right)\\
   &=o(k^{(3-\alpha)/4})=o(k^{(3-\alpha)/2}),
\endaligned$$
contradiction with (\ref{eqn:3d_im}).
\end{remark}

\begin{lemma}\label{lemma:fg_central}
Suppose that $f$ and $g$ are two nonnegative functions on
$\zz{Z}^d$, and $f$ has compact support. Suppose further that both
$f(x)$ and $g(x)$ are decreasing functions in $|x|$, then
\begin{equation}
   \label{ie:fg_central}\sum_y f(y) g(x-y) \leq \sum_y f(y) g(y),\,\forall\, x.
\end{equation}
\end{lemma}
\begin{proof}
Since $f(x)$ is a decreasing functions in $|x|$ and has compact
support, we can enumerate its positive values, say $a_1\geq
\ldots\geq a_n>0$. We can also enumerate the values of $g$, say
$b_1\geq \ldots b_n \geq \ldots.$ To show (\ref{ie:fg_central}),
it then suffices to show that
$$
   \sup_{i_1,\ldots, i_n} \sum_{k=1}^n a_k b_{i_k} = \sum_{k=1}^n a_k b_k.
$$
But this is easily seen to be true.
\end{proof}

\section{Proof of Theorem \ref{thm:SIR}: Spatial epidemics in
Dimensions $d=2,3$} \label{sec:SIR}

\subsection{Strategy}\label{ssec:strategySIR} The strategy is the same
as that used by \cite{lalley07} in the $1-$dimensional case: Since the
law of the \emph{SIR-}$d$ epidemic with village size $N$ is absolutely
continuous relative to that of its branching envelope, and since the
branching envelopes converge weakly, after renormalization, to
super-Brownian motion, it suffices to prove that the likelihood ratios
converge weakly to the likelihood ratio \eqref{eq:RN} of the
appropriate Dawson-Watanabe process relative to super-Brownian motion.
The one- and higher-dimensional cases differ only in the behavior of
the occupation statistics that enter into the likelihood ratios.

\subsection{Modified \emph{SIR}-$d$ epidemic}\label{ssec:collisions}
As in the one-dimensional case, it is technically easier to work
with the likelihood ratio for a modification of the \emph{SIR}-
$d$ epidemic.  Recall that (a) when an infected individual
attempts to infect a recovered individual in an \emph{SIR}
epidemic,  the attempt fails; and (b) when two (or more) infected
individuals simultaneously attempt to infect the same susceptible
individual, all but one of the attempts fail. Call an occurrence
of type (a) an \emph{errant attempt}, and an occurrence of type
(b) a \emph{collision}. In the modified \emph{SIR} epidemic,
collisions are not allowed, and there can be at most one errant
attempt at any site/time. A formal specification of the modified
\emph{SIR} epidemic uses a variation of the standard coupling
described in section~\ref{ssec:bre}, as follows:

\medskip \noindent \textbf{Modified Standard Coupling:} Particles are
colored \emph{red} or \emph{blue}; red particles represent
infected individuals in the modified \emph{SIR} epidemic. Each
particle produces a random number of offspring, according to the
Poisson($1$) distribution, which then randomly move to neighboring
sites. Once situated, these offspring are assigned colors
according to the following rules:
\begin{itemize}
\item [(A)] Offspring of blue particles are  blue; offspring of red
particles may be either red or blue.
\item [(B)] At any site/time $(x,t)$ there is at most one blue
offspring of a red parent.
\item [(C)] Given that at site $x$ and time $t$ there are $y$ offspring
of red parents, the conditional probability
$\kappa_{N}(y)=\kappa_{N,t,x}(y)$ that one of them is blue is
\begin{align}\label{eq:kappa}
    \kappa_{N} (y)&=\{yR/N \}\wedge 1, \quad \text{where}\\
\label{eq:numberTagsUsed}
        R&=R^{N}_{t} (x) =\sum_{s<t}Y^{N}_{s} (x)
\end{align}
and $Y^{N}_{t} (x)$ is the number of \emph{red} particles at site $x$
in generation $t$. (Thus, $R=R^{N}_{t} (x)$ is the number of recovered
individuals at site $x$ at time $t$.) The red particle process is the
\emph{modified SIR epidemic}.
\end{itemize}

\begin{prop}\label{proposition:modifiedSIR}
For each $N\geq 1$, versions of the \emph{SIR} epidemic and the
\emph{modified} \emph{SIR} epidemic can be constructed on a common
probability space in such a way that (i) the initial configurations
$\mu^{N}$ of infected individuals are identical, and satisfy the
hypothesis \eqref{hyp:initialConfig} of Theorem~\ref{thm:SIR}; and
(ii) the discrepancy $D_{t} (x)$ between the two processes at
site $x$ and time $t$ (that is, the absolute difference in number of
infected individuals) satisfies
\begin{equation}\label{eq:discrepancy}
    \max_{t,x} D_{t} (x)=o_{P} (N^{\alpha}).
\end{equation}
\end{prop}

This implies that after Feller-rescaling, the \emph{SIR-}$d$ epidemic
and the modified \emph{SIR-}$d$ epidemic are
indistinguishable. Consequently, to prove Theorem~\ref{thm:SIR} it
suffices to prove the corresponding result for the modified epidemic.

\begin{proof}
[Proof of Proposition~\ref{proposition:modifiedSIR}] This is
essentially the same as the proof of Proposition~3 in \cite{lalley07},
except that different estimates for the numbers of collisions and
errant infection attempts in the \emph{SIR}-$d$ epidemic are
necessary. These are given in Lemma~\ref{lemma:collisions} below.
\end{proof}

\begin{lemma}\label{lemma:collisions}
For each pair $(n,x)\in \zz{N}\times \zz{Z}^{d}$, let
$\Gamma^{N}_{n} (x)$ and $A^{N}_{n} (x)$ be the number of
collisions and the number of errant infection attempts,
respectively, at site $x$ and time $n$ in the \emph{SIR-}$d$
epidemic with village size $N$. Assume that the hypotheses
\eqref{hyp:initialConfig}-\eqref{hyp:smoothness} of
Theorem~\ref{thm:SIR} are satisfied, for some $\alpha \leq 1/
(3-d/2)$. Then
\begin{equation}\label{eq:collisionEstimate}
    \sum_{n}\sum_{x} \left\{\Gamma^{N}_{n} (x) + (A^{N}_{n} (x)-1)_{+}
    \right\}
    =o_{P} (N^{\alpha}).
\end{equation}
\end{lemma}
The proof of this lemma makes use of the following result.
\begin{lemma}\label{lemma:un2}[Proposition 28 in \cite{lz07}] Denote by
$U_n(x)$ the number of particles at $x$ at time $n$ of a BRW
started by one particle at the origin, then
\begin{equation}\label{eqn:un2nd}
   EU_n(x)^2=P_n(x)+\sigma^2\sum_{i=0}^{n-1}\sum_z P_i(z)P_{n-i}^2(x-z),
\end{equation}
where $\sigma^2$ is the variance of the offspring distribution.
\end{lemma}

\begin{proof}[Proof of Lemma \ref{lemma:collisions}]
Since the life length of the process is $O_p(N^\alpha)$, it
suffices to show that for any $t>0$,
\[
  \sum_{n\leq N^\alpha t}\sum_{x} \left\{\Gamma^{N}_{n} (x) + (A^{N}_{n} (x)-1)_{+}
    \right\}
    =o_{P} (N^{\alpha}).
\]

Consider first the number $\Gamma^{N}_{n} (x)$ of collisions at
site $x$ and time $n$.  For any susceptible individual $\eta$, a
collision occurs at $\eta$ if and only if there is some pair $\xi
,\zeta$ of infected individuals at neighboring sites that
simultaneously attempt to infect $\eta$. Therefore given the
evolution up to time $n$,  the conditional expectation of
$\Gamma^{N}_{n+1} (x)$ is bounded by $C(\sum_{e}X_n^N(x+e))^2/N $.
We want to show that
\begin{equation}\label{eq:collisionEst}
   \sum_{n\leq N^\alpha t}\sum_x  (X_n^N(x))^2/N    =o_p(N^{\alpha}).
\end{equation}
By the dominance of BRW over SIR epidemic, if we denote by
$U_n(x)$ the number of particles at $x$ at time $n$ of a BRW
started by one particle at the origin, and $x_i, i=1,2,\ldots,$
the position of the initial particles of our epidemic model,  then
$$
   E(X_n^N(x))^2 \leq \sum_i EU_n(x-x_i)^2 + 2\sum_{i\neq j} P_n(x-x_i)P_n(x-x_j),
$$
which, in dimension 3, by Lemma \ref{lemma:un2}, can be bounded by
$C\sum_i P_n(x-x_i)+ 2\sum_{i\neq j} P_n(x-x_i)P_n(x-x_j) $.
Therefore
\begin{equation}\label{eqn:xn2_3d}
  \aligned \sum_{n\leq N^\alpha t}\sum_x E(X_n^N(x))^2
  &\leq C\sum_{n\leq N^\alpha t}(CN^\alpha +
          CN^{2\alpha}/\sqrt{n}^3)\\
  &=O(N^{2\alpha}),\endaligned
\end{equation}
which is $o(N\times N^\alpha)$ since $\alpha \leq 2/3.$  In
dimension 2, again by Lemma \ref{lemma:un2}, $EU_n(x)^2\leq
C(1+\log n) P_n(x) $, therefore
\begin{equation}\label{eqn:xn2_2d}\aligned
     \sum_{n\leq N^\alpha t}\sum_x E(X_n^N(x))^2
     & \leq C\sum_{n\leq N^\alpha t}(C(1+\log n) N^\alpha +
             CN^{2\alpha}/n)\\
     &=O(N^{2\alpha}\log N),\endaligned
\end{equation}
which is also $o(N\times N^\alpha)$ since $\alpha \leq 1/2.$

Now consider the number $A^{N}_{n} (x)$ of errant infection
attempts at site $x$ and time $n$. In order that there be more
than one errant attempt, either (i) two or more infected
individuals must simultaneously try to infect a recovered
individual, or (ii) infected individuals must attempt to infect
more than one recovered individual. The number of occurences of
type (i) during the course of the epidemic is $o_{P}
(N^{\alpha})$, by the same argument that proved
\eqref{eq:collisionEst}. Thus, it suffices to bound the number of
errant attempts of type (ii). This is  bounded by  the number
$B^{N}_{n} (x)$ of pairs $\varrho ,\varrho '$ of recovered
individuals at site $x$ and time $n$ that are subject to
simultaneous infection attempts.  Clearly,
\[
    B^{N}_{n} (x)\leq \sum_{\xi ,\varrho } \sum_{\zeta ,\varrho' }
          Z_{\xi ,\varrho } Z_{\zeta ,\varrho '}
\]
where the sums are over all pairs $( (\xi ,\varrho ), (\zeta
,\varrho '))$ in which $\varrho ,\varrho '$ are recovered
individuals at site $x$ and time $n$ and $\xi ,\zeta$ are infected
individuals at neighboring sites, and $Z_{\xi ,\varrho}$ and
$Z_{\zeta ,\varrho '}$ are independent Bernoulli($1/((2d+1)N) $).
Hence,
\[
    E (B^{N}_{n+1} (x)\,|\, \mathcal{G}_{n})
      \leq C\left(\sum_{e} X_n^N(x+e)\right)^2
(R_n^N(x)/N)^2.
\]
By Theorem \ref{thm:lt} and the dominance of BRW over SIR
epidemic, for all $\epsilon>0,$ there exists $C>0$ such that with
probability $\geq 1-\epsilon$,
$$
   \max_x R_{N^\alpha t}^N(x)\leq C N^{\alpha(2-d/2)}.
$$
Note further that
$$
   \sum_{n\leq N^\alpha t}\sum_x E\left(\sum_{e} X_n^N(x+e)\right)^2
   \leq C\sum_{n\leq N^\alpha t}\sum_x E(X_n^N(x))^2,
$$
which, by (\ref{eqn:xn2_2d})  and (\ref{eqn:xn2_3d}), is bounded
by $C N^{2\alpha}\log N$ in dimension 2 and $C N^{2\alpha}$ in
dimension 3. Therefore, by enlarging $C$ if necessary we have that
with probability $\geq 1-2\epsilon$, the following holds:
$$
   \sum_{n\leq N^\alpha t}\sum_x\left(\sum_{e} X_n^N(x+e)\right)^2
          \left(R_n^N(x)/N\right)^2
   \leq CN^{2\alpha(2-d/2)} N^{2\alpha}\log N /N^2
   =o(N^\alpha).
$$
\end{proof}

\subsection{Convergence of likelihood ratios}\label{ssec:convergenceLR}
In view of Proposition \ref{proposition:modifiedSIR}, to prove
Theorem \ref{thm:SIR} it suffices to prove the corresponding
result for the modified \emph{SIR}  epidemic defined in \S
\ref{ssec:collisions}. For this, we shall analyze   likelihood
ratios. Denote by $Q^{N}$ the law of the modified \emph{SIR}
epidemic, and by $P^{N}$ the law of the branching envelope. Recall
(\hbox{cf.} the \emph{modified standard coupling}) that in the
modified \emph{SIR} process there can be at most one errant
infection attempt, and no collisions, at any site/time $x,t$.
Given the evolution of the process up to time $t-1$, infection
\emph{attempts} at site $x$ and time $t$ are made according to the
same law as are offspring in the branching envelope; the
conditional probability that one of the attempts is errant is
$\kappa_{N} (y)$ (see equation \eqref{eq:kappa}). Consequently,
the likelihood ratio $dQ^{N}/dP^{N}$ at the sample evolution
$X^{N}:=\{X^{N}_{t} (x) \}_{x,t}$ is
\begin{equation}\label{eq:LR}
    \frac{dQ^{N}}{dP^{N}}
    =\prod_{t\geq 1} \prod_{x\in \zz{Z}^{d}}
    \frac{p(y|\lambda)(1-\kappa_N(y)) + p(y+1|\lambda)\kappa_N(y+1)}{p(y|\lambda)},
\end{equation}
where
\begin{gather*}
    y=X_{t}^N(x),\\
    \lambda =\lambda^{N}_{t} (x)=
        \sum_{e} X^{N}_{t-1}(x+e)/(2d+1), \quad \text{and}\\
    p (k\,|\, \lambda)=\lambda^{k}e^{-\lambda}/k!
\end{gather*}
 By the same calculation as in \cite{lalley07}, equation (53), this
can be rewritten as
\begin{equation}\label{eq:LRrewritten}
    \frac{dQ^{N}}{dP^{N}}
  =(1+\varepsilon_N)\exp\left\{-\sum_{t}\sum_x
         \Delta_{t}^N(x)\varrho_{t}^N(x) - \frac{1}{2}\sum_{t}\sum_x
         \Delta_{t}^N(x)^2\varrho_{t}^N(x)^2\right\},
\end{equation}
where
$$\aligned \Delta_{t}^N(x)&:=(X_{t}^N(x)-\lambda_{t}^N(x))/N^\alpha,\\
          \varrho_{t}^N(x)&:=R_{t}^N(x)/N^{1-\alpha}; \quad \text{and}\\
          \varepsilon_N &=o_p(1) \quad \text{under} \; P^{N}.
\endaligned$$
That the error term $\varepsilon_N$ is $o_p(1)$  follows by an
argument nearly identical to  the proof of Lemma
\ref{lemma:collisions}.

Observe that under $P^{N}$, the increments (in $t$) of the first sum
in the exponential constitute a martingale difference
sequence. Furthermore, the quantities $\Delta^{N}_{t} (x)$ in equation
\eqref{eq:LRrewritten} are the atoms of the \emph{orthogonal
martingale measures} $M^{N}$ associated with the branching random
walks $X^{N}$. See \cite{lalley07} for the analogous representation in
the one-dimensional case, and \cite{Walsh86} for background on
stochastic integration against orthogonal martingale measures. The
martingale measures $M^{N}$ can be defined by their actions on test
functions $\psi \in C^{\infty}_{c} (\zz{R}^{d})$. Write $\ip{\mu
,\psi}$ for the integral of $\psi$ against a finite Borel measure
$\mu$ on $\zz{R}^{d}$, and $\mathcal{F}_{k}$ for the Feller-Watanabe
rescaling operator \eqref{eqn:feller_scl}; then
\[
    M^{N}_{t} (\psi)= \ip{\mathcal{F}_{N^{\alpha}}X^{N}_{N^\alpha t},\psi }
                -\ip{\mathcal{F}_{N^{\alpha}}X^{N}_{0},\psi}
                -\int_{0}^{t}
                \ip{\mathcal{F}_{N^{\alpha}}X^{N}_{N^\alpha s},A_{N^{\alpha
                }}\psi } \, ds
\]
where
$A_{k}$ is the difference operator
\[
    A_k\psi(x)
   =\left(\sum_{e}\psi(x+e/\sqrt{k})-(2d+1)\psi(x)\right)/
    \left[(2d+1)k^{-1}\right].
\]
The first sum in the exponential of equation \eqref{eq:LRrewritten}
can  be expressed as a stochastic integral against the orthogonal
martingale measure $M^{N}$:
\begin{equation}\label{eq:stochasticIntegralRep}
    \sum_{t\geq 1}\sum_{x\in \zz{Z}^{d}} \Delta_{t}^N(x)\varrho_{t}^N(x)
   =\int\int \theta^N(t,x) M^N\,(dt,dx),
\end{equation}
where
\[
    \theta^N(t,x)=R_{N^\alpha t}^N(\sqrt{N^\alpha}x)/N^{1-\alpha}.
\]

\begin{prop}\label{prop:omg_conv}
Let $X$ be the Dawson-Watanabe process with initial configuration
$\mu$ and variance parameter $\sigma^{2}$, and let $M (dt, dx)$ and
$L_{t} (x)$  be the associated orthogonal martingale measure and local
time density process. Then under $P^{N}$, given the hypotheses of
Theorem~\ref{thm:SIR}, as $N \rightarrow \infty$,
\begin{align*}
    (\mathcal{F}_{N^{\alpha}}X^{N},\theta^{N},M^{N})
    & \Longrightarrow (X,0,M) \quad \text{if} \; \; \alpha <1/(3-d/2)
        \quad \text{and}\\
    (\mathcal{F}_{N^{\alpha}}X^{N},\theta^{N},M^{N})
    & \Longrightarrow (X,L,M) \quad \text{if} \; \; \alpha =1/ (3-d/2) .
\end{align*}
\end{prop}

\begin{proof}
Given the weak convergence of the second margin $\theta^{N}$, the
joint convergence of the triple follows by the same argument as in
Proposition 4 of \cite{lalley07}.  The asymptotic behavior of the
processes $\theta^{N}$ follows from Theorem~\ref{thm:lt}.
\end{proof}

\begin{cor}\label{corollary:subThreshold}
If $\alpha <1/(3-d/2)$ then under $P^{N}$, as $N \rightarrow \infty$,
\begin{equation}\label{eq:lrConvsubThresh}
    \frac{dQ^{N}}{dP^{N}} \longrightarrow 1 \quad \text{in
    probability}
\end{equation}
provided that the hypotheses of
Theorem~\ref{thm:SIR} on the initial configurations are satisfied.
\end{cor}

\begin{proof}
Proposition \ref{prop:omg_conv} implies that the sums
\eqref{eq:stochasticIntegralRep} converge to zero in probability as $N
\rightarrow \infty$. That the second sum in the likelihood ratio
\eqref{eq:LRrewritten}  also converges to zero in probability follows
by the same argument as in the one-dimensional case (see the proof of
equation (60) in \cite{lalley07}).
\end{proof}

\begin{proof}
[Proof of Theorem \ref{thm:SIR}] Corollary
\ref{corollary:subThreshold} implies that the modified \emph{SIR}
epidemics have the same scaling limit as their branching envelopes
when $\alpha <1/ (3-d/2)$. Thus, to complete the proof of
Theorem~\ref{thm:SIR}, it suffices to prove the assertion
\eqref{eq:SIR-convergence} when $\alpha =1/ (3-d/2)$. For this, it
suffices to show that the two sums in the exponential of
equation~\eqref{eq:LRrewritten}  converge to the corresponding
integrals in the exponential of equation~\eqref{eq:RN}. The
convergence of the first sum follows from
Proposition~\ref{prop:omg_conv} and the representation
\eqref{eq:stochasticIntegralRep}. By the same argument as in the
proof of Corollary~4 of \cite{lalley07},
\begin{equation}\label{eq:SIConvergence}
    \sum_n\sum_x \Delta_n^N(x)\varrho_n^N(x)
    =\int\int \theta^N(t,x) M^N(dt\,dx)\Rightarrow\int\int L_t(x)M(dt\,dx).
\end{equation}

 The convergence  of the second sum
\begin{equation}\label{eqn:QV_conv}
  A^{N}:=\sum_n\sum_x \Delta_n^N(x)^2\varrho_n^N(x)^2
  \Rightarrow \int \la X_t, (L_t)^2\ra\, dt
\end{equation}
follows by an argument similar to the proof of equation (60) in
\cite{lalley07}. The idea is that if one substitutes the conditional
expectation $\lambda^{N}_{n} (x)/N^{2\alpha }=E (\Delta_n^N(x)^2\,|\,
\mathcal{G}_{n-1})$ for the quantity $\Delta_n^N(x)^2$ in the sum
\eqref{eqn:QV_conv} then the modified sum converges; in particular,  by
Theorem~\ref{thm:lt} and Watanabe's theorem,
$$\aligned
   B^N
   &:=\sum_n\sum_x\lambda^{N}_{n} (x) /N^{2\alpha}\times [R_n^N(x)/N^{1-\alpha}]^2\\
   &=\frac{1}{N^\alpha}\sum_n\sum_x [\sum_e X_{n-1}^N(x+e)/(2d+1)]/N^{\alpha} \times
         [R_n^N(x)/N^{\alpha(2-d/2)}]^2\\
   & \Longrightarrow  \int \la X_t, (L_t)^2\ra \,dt,
\endaligned$$
where the second equation holds because $\alpha =1/ (3-d/2)$.
Therefore, it suffices to show that replacing $\Delta_n^N(x)^2$ by
its conditional expectation has an asymptotically negligible
effect on the sum, that is,
$$
  A^N-B^N=o_p(1).
$$
By a simple variance calculation (see \cite{lalley07} for the
one-dimensional case), this reduces to
proving that
\begin{equation}\label{eq:finalObj}
    \sum_n \sum_x (\lambda_n^N(x))^2/N^{4\alpha}\times
   [R_n^N(x)/N^{\alpha(2-d/2)}]^4
= o_p(1).
\end{equation}
In fact, for all $\epsilon>0,$ there exists $C>0$ such that with
probability $\geq 1-\epsilon$,
$$
   \max_x R_{N^\alpha t}^N(x)\leq C N^{\alpha(2-d/2)}.
$$
Note further that
$$
    \sum_{n\leq N^\alpha t}\sum_x E[\sum_{e} X_n^N(x+e)]^2
    \leq C\sum_{n\leq N^\alpha t}\sum_x E(X_n^N(x))^2,
$$
which, by (\ref{eqn:xn2_2d})  and (\ref{eqn:xn2_3d}),  is bounded
by $C N^{2\alpha}\log N$ in dimension 2 and $C N^{2\alpha}$ in
dimension~3. Therefore, by enlarging $C$ if necessary we have that
with probability $\geq 1-2\epsilon$, the following holds:
$$
   \sum_{n\leq N^\alpha t}\sum_x(\lambda_n^N(x))^2/N^{4\alpha}
      \times [R_n^N(x)/N^{\alpha(2-d/2)}]^4
   \leq C N^{2\alpha}\log N /N^{4\alpha}
   =o(1).
$$
\end{proof}

\section{Appendix: Proofs of Lemmas \ref{lemma:lltA}-\ref{lemma:greenB}}\label{sec:app}

\subsection{Proofs of Lemmas
\ref{lemma:lltA}--\ref{lemma:lclt_diff}}\label{ssec:lltA}  The
strategy is to consider the regions $|x|\leq (2L n \log n)^{1/2}$
and $|x|\geq (L n \log n)^{1/2}$ separately. We begin with the
unbounded region. By Hoeffding's inequality, since the increments
of $S_{n}$ are no larger than~$1$ in modulus,
$$
     P_n(x)\leq P(|S_n|\geq |x|)\leq 2d\exp(-|x|^2/(2dn)).
$$
Now for $0<\beta < 1/\sqrt{d}$ and $L=L (\beta)$ sufficiently
large,
\[
    \exp(-|x|^2/(2dn)) \leq \exp(-\beta^2 |x|^2/(2n))/n^{(d+1)/2},
    \quad \forall \; |x|\geq \sqrt{L n \log n}.
\]
Thus,
\begin{equation}\label{eqn:pn_lgdv}
   P_n(x)\leq C \exp(-\beta^2 |x|^2/(2n))/n^{(d+1)/2}=C\phi_n(\beta
   x)/\sqrt{n}, \quad \forall\, |x|\geq \sqrt{L n \log n},
\end{equation}
and
\begin{align}\label{eqn:p_diff_aty}
     |P_n(x)-P_n(y)|&\leq C\Phi_n(\beta x,\beta y)/\sqrt{n} \\
\notag               &\leq  C\left(\frac{|x-y|}{\sqrt{n}}\wedge
                 1\right)
                 \Phi_n(\beta x,\beta  y), \quad
                 \forall \, |x|,|y|\geq \sqrt{L n \log n}.
\end{align}
This proves inequalities \eqref{eqn:lclt_bd} and
\eqref{eqn:p_diff} for $x$ and $y$ outside the ball of radius $(L
n \log n)^{1/2}$.

To deal with the region $|x|\leq
 (2L n \log n)^{1/2}$ we shall use the following crude estimate, valid
for all points $x\in \zz{Z}^{d}$ (Theorem 2.3.5 in
\cite{lawler07}):
$$
     |P_n(x)- \sigma^{-d}\phi_{n} (x/\sigma ) |\leq C/(\sqrt{n}^d\cdot n).
$$
For $\beta =\beta (L)>0$ sufficiently small,
\[
    \phi_n(\beta x)\geq 1/(\sqrt{n}\cdot n^{d/2}), \quad
             \forall  \; |x|\leq \sqrt{2L n \log n};
\]
consequently,
\begin{equation}\label{eqn:ppbar_diff}
     |P_n(x)- \sigma^{-d}\phi_{n} (x/\sigma )|\leq C \phi_n(\beta
     x)/\sqrt{n},
     \quad  \forall\, |x|\leq \sqrt{2L n \log n}.
\end{equation}
This obviously implies  \eqref{eqn:lclt_bd}  for $x$ in the region
$|x|\leq  (2L n \log n)^{1/2}$, and hence, together with the
argument of the preceding paragraph, completes the proof of
\eqref{eqn:lclt_bd}.

Similar arguments can be used to establish inequality
\eqref{eqn:p_diff} for points $x$ and $y$ in the ball of radius $
(L n \log n)^{1/2}$ centered at the origin. First, it is easily
seen that for sufficiently small $\beta >0$,
\[
    |\phi_{n} (x)-\phi_{n} (y)| \leq
     C\left((|x-y|/\sqrt{n})\wedge 1\right)  \Phi_n(\beta  x,\beta
     y),
     \quad \forall  \;x,y \in \zz{R}^{d}
\]
Hence, by \eqref{eqn:ppbar_diff}, \eqref{eqn:p_diff} holds for for
$x$ and $y$ in the ball of radius $ (L n \log n)^{1/2}$.
Therefore, since the choice of $L$ is arbitrary, to complete the
proof of (\ref{eqn:p_diff}) it suffices to consider the case where
$|x|\leq (L n \log n)^{1/2} $ and $|y|\geq (2L n \log n)^{1/2}$.
In this case, choose a point $z$ in the annulus $|z|\in ( (L n
\log n)^{1/2}, (2L n \log n)^{1/2})$ such that $|x-z| + |z-y|\leq
2|x-y|$. Using the fact that (\ref{eqn:p_diff}) holds for each of
the pairs $x,z$ and $z,y$, we have
\[
     |P_n(x)-P_n(z)|\leq C\left((|x-z|/\sqrt{n})\wedge 1\right)
    \Phi_n(\beta x,\beta z)
    \leq 2C \left((|x-z|/\sqrt{n})\wedge 1\right) \Phi_n(\beta
    x,\beta y)
\]
and
\[
     |P_n(z)-P_n(y)|\leq C\left(|z-y|/\sqrt{n}\wedge 1\right)
     \Phi_n(\beta z,\beta y)
     \leq C\left((|z-y|/\sqrt{n})\wedge 1\right)
     \Phi_n(\beta x,\beta y).
\]
Consequently,
$$
     |P_n(x)-P_n(y)|
     \leq
      C \left((|x-y|/\sqrt{n})\wedge 1\right)\cdot\Phi_n(\beta
      x,\beta y).
$$
This completes the proof of  \eqref{eqn:p_diff}. Inequality
\eqref{p_diff_alpha} obviously follows from \eqref{eqn:p_diff}.

\bigskip \noindent
\textbf{Proof of \eqref{eqn:green_bd} when $d=3$.} The following
argument works for all $d\geq 3$. Firstly, $\sum_{n\leq kT}
\phi_n(\beta x)$  is bounded by  $\sum_{n=1}^\infty \phi_n(\beta
x)$. This is a decreasing function in $|x|$; moreover, by Lemma
4.3.2 in \cite{lawler07}, it equals $C_1/|x|^{d-2} +
O(1/|x|^{d+2})$ as $|x|\to \infty$ for some $C_1>0$. Secondly, for
all~$k$ sufficiently large and all $|x|\leq A\sqrt{k}$,
$$
  \sum_{n\leq kT}   P_n(x)\geq \sum_{kT/2\leq n\leq kT} P_n(x)\geq \sum_{kT/2\leq n\leq kT}
   n^{-d/2} C\geq Ck^{1-d/2};
$$
note further that
$$
 \sum_{n> kT}   P_n(x)\leq C\sum_{n> kT} n^{-d/2}\leq
 Ck^{1-d/2};
$$
therefore there exists $\delta>0$ such that all $k$ sufficiently
large and all $|x|\leq A\sqrt{k}$,
$$\sum_{n\leq kT} P_n(x) \geq \delta \sum_{n=1}^\infty P_n(x).$$
For the nearest neighbor random walk,  $\sum_{n=1}^\infty P_n(x)$
equals $C_2/|x|^{d-2} + O(1/|x|^{d+2})$ as $|x|\to~\infty$ for
some $C_2>0$. Relation \eqref{eqn:green_bd} follows.

\bigskip \noindent
\textbf{Proof of \eqref{eqn:green_bd} when $d=2$.} In this case,
one can deduce from the proof of Theorem 4.4.3 in \cite{lawler07}
that there exist $C_i>0,\; i=1,2,3,4$ such that for all $|x|\leq
A\sqrt{k}$,
$$
  \sum_{n\leq kT } \phi_n(\beta x)\asymp C_1 + C_2 \log(kT/|x|^2),
$$
and
$$
  \sum_{n\leq kT } P_n(x)\asymp C_3 + C_4 \log(kT/|x|^2).
$$
\eqref{eqn:green_bd} follows.

\bigskip
To complete the proof of Lemma \ref{lemma:lltA}, it remains to
prove inequality \eqref{eqn:dis_conv}.

\bigskip \noindent
\textbf{Proof of \eqref{eqn:dis_conv}.} By \eqref{eqn:lclt_bd}, it
suffices to show that there exists $C>0$ such that for all
$x\in\zz{Z}^d$ and all $i,j\in\zz{N}$,
\begin{equation}\label{eqn:conv_1}
  \sum_y\phi_i(\beta y)\phi_j(\beta (x-y))\leq C\phi_{i+j}(\beta
  x/2).
\end{equation}
For all $y\in \zz{Z}^d$, Let $Q_y$ be the cube centered at $y$
with length 1, and define
$$
  \tilde{\phi}_i(y)=\int_{z\in Q(y)} (2\pi
  i/\beta^2)^{-d/2}\exp(-\beta^2|z|^2/(2i))dz.
$$
Then there exists $C>0$ such that for all $i$ and all $x$,
$$
  \phi_i(\beta x)\leq C\tilde{\phi}_i( x).
$$
Therefore to show \eqref{eqn:conv_1}, it suffices to show that
there exists $C>0$ such that for all $x\in\zz{Z}^d$ and all
$i,j\in\zz{N}$,
\begin{equation}\label{eqn:conv_2}
  \sum_y\tilde{\phi_i}(y)\tilde{\phi}_j(x-y)\leq C\phi_{i+j}(\beta
  x/2).
\end{equation}
Note $(\tilde{\phi}_i(\cdot))$ is the probability mass function of
the random variable $[\Lambda_i]$, where $\Lambda_i\sim
N(0,i/\beta\cdot I_d)$, and for any $z\in\zz{R}^d\backslash \cup_y
\partial Q(y)$, $[z]$ is the unique $y$ such that $z\in Q(y)$
($ \Lambda_i$ takes values on $\cup_y
\partial Q(y)$ with probability 0, so $[\Lambda_i]$ is well
defined \hbox{a.s.}). Hence
$\sum_y\tilde{\phi_i}(y)\tilde{\phi}_j(\cdot-y)$ is the
probability mass function of $[\Lambda_i]+[\Lambda_j]$ with
$\Lambda_i$ and $\Lambda_j$ being independent. Since
$[\Lambda_i]+[\Lambda_j]$ differs from $\Lambda_i+\Lambda_j$ by at
most 2,
$$
   \sum_y\tilde{\phi_i}(y)\tilde{\phi}_j(x-y)\leq \int_{|z-x|\leq
   2} (2\pi  (i+j)/\beta^2)^{-d/2}\exp(-\beta^2|z|^2/(2(i+j)))dz.
$$
It is easy to see that the last term can be bounded by
$C\phi_{i+j}(\beta   x/2)$ for some $C$ independent of $i,j$ and
$x$.

 \qed

\subsection{Proof of Lemma \ref{lemma:greenA}}\label{ssec:greenA}
By the monotonicity of $\phi_{n} (x)$ in $|x|$, for all integers
$m,l\geq 1$ and all $x, y\in\zz{R}^d$  we have
$$
\aligned
  \phi_m(x)\phi_l(y)
  &\leq \phi_m(x)\phi_l(x)+\phi_m(y)\phi_l(y)\\
  &\leq C(ml)^{-d/4}(\phi_{ml/(m+l)}(x) +\phi_{ml/(m+l)}(y)).
\endaligned
$$
Now note that for any $t>0$ and any $x$,
$$
 \phi_t(x)=(2\pi t)^{-d/2}\exp(-|x|^2/(2t))\leq 2^{d/2}\cdot \phi_{2t}(x),
$$
and when $t\geq 1$,
$$
 \phi_t(x)\leq \phi_{\lceil t \rceil}(x)\cdot (\lceil t \rceil/t)^{d/2}\leq 2^{d/2} \phi_{\lceil t \rceil}(x),
$$
where $\lceil t \rceil$ stands for the smallest integer bigger
than or equal to $t$.  Further note that when $m,l\geq 1$, $
ml/(m+l)\geq 1/2$. Using the two inequalities above we then get
\begin{equation}\label{eqn:phi_t}
  \phi_m(x)\phi_l(y)
  \leq C (ml)^{-d/4} \left(\phi_{\lceil 2ml/(m+l) \rceil}(x)+\phi_{\lceil 2ml/(m+l)
  \rceil}(y)\right).
\end{equation}
and
$$
\aligned
    \Phi_m(x,y)\Phi_l(u,v)
    &=\phi_m(x)\phi_l(u)+\phi_m(x)\phi_l(v)+\phi_m(y)\phi_l(u)+\phi_m(y)\phi_l(v)\\
    &\leq C (ml)^{-d/4} \left(\Phi_{\lceil 2ml/(m+l) \rceil}(x,y)+\Phi_{\lceil 2ml/(m+l)
  \rceil}(u,v)\right)
\endaligned
$$
Therefore for all $h_1,h_2\geq 1$,
\begin{equation}\label{eqn:digamma_prod}
\aligned
   &F_{n,h_1}(x,y;\beta)F_{n,h_2}(x,y;\beta)\\
   =&\sum_{|\rho_1|<h_1}\sum_{|\rho_2|<h_2}\sum_{m<n}  \sum_{l<n}
      (ml)^{-\gamma/2}\Phi_m(\beta(x+\rho_1),\beta(y+\rho_1))\cdot\Phi_l(\beta(x+\rho_2),\beta(y+\rho_2))      \\
   \leq& \sum_{|\rho_1|<h_1}\sum_{|\rho_2|<h_2}\sum_{m<n}\sum_{l<n} C
   (ml)^{-d/4-\gamma/2}\\
   &\quad \cdot  \left\{\Phi_{\lceil 2ml/(m+l) \rceil}(\beta(x+\rho_1),\beta(y+\rho_1))+\Phi_{\lceil 2ml/(m+l)
  \rceil}(\beta(x+\rho_2),\beta(y+\rho_2))\right\}\\
  \leq& \sum_{|\rho|<h_1+h_2-1}\cdot\sum_{m<n}\sum_{l<n}C (ml)^{-d/4-\gamma/2}\Phi_{\lceil 2ml/(m+l)
  \rceil}(\beta(x+\rho),\beta(y+\rho)).
\endaligned
\end{equation}
Observe that when $m,l\in [1,n)$, $ ml/(m+l)\in [1/2,n/2)$, hence
the last term is bounded by
$$\aligned
  &C   n^{2-(d+2\gamma)/2}\sum_{|\rho|<h_1+h_2-1}\cdot\sum_{l<n}\Phi_l(\beta(x+\rho),\beta(y+\rho))\\
  \leq& C  n^{2-(d+\gamma)/2}\sum_{|\rho|<h_1+h_2-1}\cdot
    \sum_{l<n}l^{-\gamma/2}\Phi_l(\beta(x+\rho),\beta(y+\rho))\\
  =&C  n^{2-(d+\gamma)/2}F_{n,h_1+h_2-1}(x,y;\beta),
\endaligned
$$
\hbox{i.e.}, \eqref{eqn:green_indc} holds.

We now prove \eqref{eqn:conv}. By \eqref{eqn:digamma_prod},
$$\aligned
   &\sum_{i<n}\sum_z P_i(z)
      \cdot\left[F_{n-i,h_1}(x-z,y-z;\beta)F_{n-i,h_2}(x-z,y-z;\beta)\right]\\
   &\leq \sum_{i<n}\sum_z
   P_i(z)\sum_{|\rho|<h_1+h_2-1}\sum_{m<n-i}\sum_{l<n-i}\\
   &\quad\cdot      C (ml)^{-d/4-\gamma/2}\Phi_{\lceil 2ml/(m+l)
      \rceil}(\beta(x-z+\rho),\beta(y-z+\rho))\\
   &\leq \sum_{m<n} \sum_{l<n}  C (ml)^{-d/4-\gamma/2}\!\!\sum_{|\rho|<h_1+h_2-1}\cdot \sum_{i<{\rm
   min}(n-m,n-l)}\\
   &\quad\cdot\sum_z P_i(z) \Phi_{\lceil 2ml/(m+l)
      \rceil}(\beta(x-z+\rho),\beta(y-z+\rho)).
\endaligned$$
Using relation \eqref{eqn:dis_conv} and noting that $\lceil
2ml/(m+l)\rceil \leq {\rm max}(m,l)$, we can further bound the
last term by
$$\aligned
   &Cn^{2-(d+2\gamma)/2}\sum_{|\rho|<h_1+h_2-1}\cdot\sum_{i<n}\Phi_i(\beta(x+\rho)/2,\beta(y+\rho)/2)\\
   &\leq C n^{2-(d+\gamma)/2}F_{n,h_1+h_2-1}(x,y;\beta/2).
\endaligned
$$

\qed

\subsection{Proof of Lemma \ref{lemma:greenB}}\label{ssec:greenB}
For all $h_1,h_2\geq 1$, all $x\in \zz{Z}^d$ and all integers
$m,n\geq 1$, by \eqref{eqn:phi_t},
\begin{equation}\label{eqn:gimel_prod}\aligned
  &J_{m,n,h_1}(x;\beta)J_{m,n,h_2}(x;\beta)\\
  &=\sum_{|\rho_1|<h_1}\sum_{|\rho_2|<h_2}\cdot\sum_{m\leq l_1,l_2<m+n}\phi_{l_1}(\beta(x+\rho_1))
    \phi_{l_2}(\beta(x+\rho_2))\\
  &\leq C\sum_{|\rho_1|<h_1}\sum_{|\rho_2|<h_2}\cdot\sum_{m\leq   l_1,l_2<m+n}\\
  &\quad\cdot (l_1 l_2)^{-d/4}
     \left\{\phi_{\lceil 2l_1l_2/(l_1+l_2) \rceil}(\beta(x+\rho_1))+\phi_{\lceil 2l_1l_2/(l_1+l_2)
     \rceil}(\beta(x+\rho_2))\right\}\\
  &\leq C\sum_{|\rho|<h_1+h_2-1}\cdot\sum_{m\leq l_1,l_2<m+n} (l_1 l_2)^{-d/4}\phi_{\lceil 2l_1l_2/(l_1+l_2)
  \rceil}(\beta(x+\rho)).
\endaligned\end{equation}
Note that when $l_1,l_2\in [m,m+n)$, $ \lceil
2l_1l_2/(l_1+l_2)\rceil\in [m,m+n)$, hence the last term is
bounded by
$$
  C   n^{2-d/2}\sum_{|\rho|<h_1+h_2-1}\cdot\sum_{m\leq
  l<m+n}\phi_l(\beta(x+\rho))=C   n^{2-d/2}J_{m,n,h_1+h_2-1}(x;\beta),
$$
\hbox{i.e.}, \eqref{eqn:green_indc_B} holds.

We now prove \eqref{eqn:conv_B}. By \eqref{eqn:gimel_prod},
$$\aligned
   &\sum_{i<n}\sum_z P_i(z)
      \cdot\left(J_{m,n-i,h_1}(x-z;\beta)J_{m,n-i,h_2}(x-z;\beta)\right)\\
   &\leq \sum_{i<n}\sum_z P_i(z)\sum_{|\rho|<h_1+h_2-1}\cdot\sum_{m\leq l_1,l_2< m+n-i}
      C (l_1l_2)^{-d/4}\phi_{\lceil 2l_1l_2/(l_1+l_2)
      \rceil}(\beta(x-z+\rho))\\
   &\leq \sum_{m\leq l_1,l_2<m+n} C (l_1l_2)^{-d/4}\!\!\!\sum_{|\rho|<h_1+h_2-1} \cdot\sum_{i<{\rm
   min}(m+n-l_1,m+n-l_2)} P_i(z) \phi_{\lceil 2l_1l_2/(l_1+l_2)
      \rceil}(\beta(x-z+\rho)).
\endaligned$$
Using relation \eqref{eqn:dis_conv} and noting that $\lceil
2l_1l_2/(l_1+l_2)\rceil \in [{\rm min}(l_1,l_2), {\rm
max}(l_1,l_2)]$, we can further bound the last term by
$$
    Cn^{2-d/2}\sum_{|\rho|<h_1+h_2-1}\cdot\sum_{m\leq i<m+n} \phi_i(\beta(x+\rho)/2)
    = Cn^{2-d/2}J_{m,n,h_1+h_2-1}(x;\beta/2).
$$

\qed

\bibliographystyle{asa}
\bibliography{mainbib}

\begin{thebibliography}{20}
\newcommand{\enquote}[1]{``#1''}
\expandafter\ifx\csname natexlab\endcsname\relax\def\natexlab#1{#1}\fi

\bibitem[{Adler(1993)}]{adler93}
Adler, R.~J. (1993), \enquote{Superprocess local and intersection local times
  and their corresponding particle pictures,} in \textit{Seminar on Stochastic
  Processes, 1992 (Seattle, WA, 1992)}, Boston, MA: Birkh\"auser Boston,
  vol.~33 of \textit{Progr. Probab.}, pp. 1--42.

\bibitem[{Aldous(1997)}]{Aldous97}
Aldous, D. (1997), \enquote{Brownian excursions, critical random graphs and the
  multiplicative coalescent,} \textit{Ann. Probab.}, 25, 812--854.

\bibitem[{Dawson(1978)}]{Dawson78}
Dawson, D.~A. (1978), \enquote{Geostochastic calculus,} \textit{Canad. J.
  Statist.}, 6, 143--168.

\bibitem[{Dawson and Hochberg(1979)}]{DH79}
Dawson, D.~A. and Hochberg, K.~J. (1979), \enquote{The carrying dimension of a
  stochastic measure diffusion,} \textit{Ann. Probab.}, 7, 693--703.

\bibitem[{Dawson and Perkins(1999)}]{DP99}
Dawson, D.~A. and Perkins, E.~A. (1999), \enquote{Measure-valued processes and
  renormalization of branching particle systems,} in \textit{Stochastic partial
  differential equations: six perspectives}, Providence, RI: Amer. Math. Soc.,
  vol.~64 of \textit{Math. Surveys Monogr.}, pp. 45--106.

\bibitem[{Dolgoarshinnykh and Lalley(2006)}]{LD06}
Dolgoarshinnykh, R.~G. and Lalley, S.~P. (2006), \enquote{Critical scaling for
  the {SIS} stochastic epidemic,} \textit{J. Appl. Probab.}, 43, 892--898.

\bibitem[{Etheridge(2000)}]{etheridge}
Etheridge, A.~M. (2000), \textit{An introduction to superprocesses}, vol.~20 of
  \textit{University Lecture Series}, Providence, RI: American Mathematical
  Society.

\bibitem[{Evans and Perkins(1995)}]{EP95}
Evans, S.~N. and Perkins, E.~A. (1995), \enquote{Explicit stochastic integral
  representations for historical functionals,} \textit{Ann. Probab.}, 23,
  1772--1815.

\bibitem[{Fleischmann(1988)}]{Fleischmann88}
Fleischmann, K. (1988), \enquote{Critical behavior of some measure-valued
  processes,} \textit{Math. Nachr.}, 135, 131--147.

\bibitem[{Iscoe(1986)}]{Iscoe86b}
Iscoe, I. (1986), \enquote{Ergodic theory and a local occupation time for
  measure-valued critical branching {B}rownian motion,} \textit{Stochastics},
  18, 197--243.

\bibitem[{Konno and Shiga(1988)}]{ks88}
Konno, N. and Shiga, T. (1988), \enquote{Stochastic partial differential
  equations for some measure-valued diffusions,} \textit{Probab. Theory Related
  Fields}, 79, 201--225.

\bibitem[{Lalley(2007)}]{lalley07}
Lalley, S. (2007), \enquote{Spatial Epidemics: Critical Behavior in One
  Dimension,} to appear in \textit{Probab. Theory Related Fields,
  arXiv:math/0701698v2}.

\bibitem[{Lalley and Zheng(2007)}]{lz07}
Lalley, S.~P. and Zheng, X. (2007), \enquote{Occupation Statistics of Critical
  Branching Random Walks,} \textit{arXiv:0707.3829v1}.

\bibitem[{Lawler and Limic(2007)}]{lawler07}
Lawler, G.~F. and Limic, V. (2007), \textit{Symmetric Random Walk},
  http://www.math.uchicago.edu/~lawler/srwbook.ps, 1st ed.

\bibitem[{Martin-L{\"o}f(1998)}]{Martinlof98}
Martin-L{\"o}f, A. (1998), \enquote{The final size of a nearly critical
  epidemic, and the first passage time of a {W}iener process to a parabolic
  barrier,} \textit{J. Appl. Probab.}, 35, 671--682.

\bibitem[{Reimers(1989)}]{Reimers89}
Reimers, M. (1989), \enquote{One-dimensional stochastic partial differential
  equations and the branching measure diffusion,} \textit{Probab. Theory
  Related Fields}, 81, 319--340.

\bibitem[{Spitzer(1976)}]{spitzer76}
Spitzer, F. (1976), \textit{Principles of random walks}, New York:
  Springer-Verlag, 2nd ed., graduate Texts in Mathematics, Vol. 34.

\bibitem[{Sugitani(1989)}]{Sugitani89}
Sugitani, S. (1989), \enquote{Some properties for the measure-valued branching
  diffusion processes,} \textit{J. Math. Soc. Japan}, 41, 437--462.

\bibitem[{Walsh(1986)}]{Walsh86}
Walsh, J.~B. (1986), \enquote{An introduction to stochastic partial
  differential equations,} in \textit{\'Ecole d'\'et\'e de probabilit\'es de
  Saint-Flour, XIV---1984}, Berlin: Springer, vol. 1180 of \textit{Lecture
  Notes in Math.}, pp. 265--439.

\bibitem[{Watanabe(1968)}]{watanabe68}
Watanabe, S. (1968), \enquote{A limit theorem of branching processes and
  continuous state branching processes,} \textit{J. Math. Kyoto Univ.}, 8,
  141--167.

\end{thebibliography}

\end{document}